\documentclass{article}
\usepackage[centering,margin=1.5cm]{geometry}

\usepackage[utf8]{inputenc}
\usepackage{blindtext}
\usepackage{setspace}
\usepackage{graphicx}
\usepackage{notoccite} 
\usepackage{lscape} 
\usepackage{caption} 
\usepackage{amsmath,amssymb,amsthm}
\usepackage{comment}
\usepackage{listings}
\usepackage{verbatim}
\usepackage{xcolor}
\usepackage{algorithm}
\usepackage{algpseudocode}
\usepackage{textcomp}
\usepackage{enumitem}
\usepackage{stmaryrd}
\usepackage{cases,xspace}
\usepackage{hyperref}%
\hypersetup{colorlinks=true, linkcolor=blue, breaklinks=true, urlcolor=blue, citecolor=blue}
\usepackage{subcaption}
\usepackage{multirow}
\usepackage{multicol}
\usepackage{mdframed}

\usepackage{tikz}    
\usetikzlibrary{matrix, positioning, fit, calc}
\usepackage{version}

\DeclareCaptionLabelFormat{figure}{{Figure #2}}
\usepackage{siunitx}

\newtheorem{definition}{Definition}
\newtheorem{thm}{Theorem}
\newtheorem{lem}{Lemma}
\newtheorem{pro}{Proposition}
\newtheorem{cor}{Corollary}
\newtheorem{remark}{Remark}

\newtheorem*{assumption*}{Assumption}
\newmdtheoremenv{assumptionbox}{Assumption}
\newenvironment{rem}{\begin{remark}}{\end{remark}}
\newcommand{\bd}{\begin{definition}} 
\newcommand{\ed}{\end{definition}} 
\newcommand{\bp}{\begin{pro}} 
\newcommand{\ep}{\end{pro}} 
\newcommand{\bt}{\begin{thm}} 
\newcommand{\et}{\end{thm}}
\newcommand{\blm}{\begin{lem}} 
\newcommand{\elm}{\end{lem}}

\newcommand{\R}{\mathbb{R}}

\newcommand{\mcX}{\mathcal{X}}

\newcommand{\mcC}{\mathcal{C}}

\newcommand{\mcS}{\mathcal{S}}

\newcommand{\realextended}{\overline{\R}}
\newcommand{\bi}{\begin{itemize}} 
\newcommand{\ei}{\end{itemize}} 
\newcommand{\bds}{\begin{description}} 
\newcommand{\eds}{\end{description}} 
\newcommand{\beq}{\begin{equation}} 
\newcommand{\eeq}{\end{equation}} 

\newcommand{\vecto}{\operatorname{vec}}

\newcommand{\abs}[1]{\left|#1\right|}

\newcommand{\prox}[1]{\mathrm{prox}_{#1}}

\newcommand{\diag}[1]{[#1]}

\newcommand{\subscr}[2]{#1_{\textup{#2}}}

\newcommand{\setdef}[2]{\{#1 \; | \; #2\}}

\newcommand{\map}[3]{#1\colon #2 \rightarrow #3}

\newcommand{\Hess}{\nabla^2}

\DeclareMathOperator*{\argmin}{arg\,min}

\newcommand{\ds}{\displaystyle}

\newcommand{\xstar}{x^{\star}}
\newcommand{\zstar}{z^{\star}}
\newcommand{\e}{\mathrm{e}}
\newcommand{\norm}[1]{\|#1\|}

\newcommand{\softt}[2]{\operatorname{soft}_{#1}\bigl({#2}\bigr)}
\newcommand{\soft}[1]{\operatorname{soft}_{#1}}
\newcommand{\sign}[1]{\operatorname{sign}\bigl({#1}\bigr)}
\newcommand{\relu}{\operatorname{ReLU}}
\newcommand{\sat}[2]{\operatorname{sat}_{#1}\bigl({#2}\bigr)}

\newcommand{\1}{\mbox{\fontencoding{U}\fontfamily{bbold}\selectfont1}}
\newcommand{\0}{\mbox{\fontencoding{U}\fontfamily{bbold}\selectfont0}}

\newcommand{\amin}{\subscr{a}{min}}
\newcommand{\amax}{\subscr{a}{max}}
\newcommand{\xmin}{\subscr{x}{min}}
\newcommand{\xmax}{\subscr{x}{max}}
\newcommand{\mf}{\rho}

\newcommand{\linexp}[1]{\operatorname{lin-exp}({#1})}

\newcommand{\odeflowtx}[2]{\phi_{#1}\bigl({#2}\bigr)}

\newcommand{\radius}{r}

\newcommand{\tld}{\subscr{t}{cross}}

\newcommand{\ce}{\subscr{c}{exp}}
\newcommand{\cl}{\subscr{c}{lin}}
\newcommand{\loc}{\textup{L}}
\newcommand{\glo}{\textup{G}}

\newcommand{\ball}[2]{B_{#1}{\bigl(#2\bigr)}}
\newcommand{\PI}{\textup{PI}~}

\definecolor{gnred}{RGB}{255,91,89}

\definecolor{gnred1}{RGB}{71,0,0} 
\definecolor{gnred2}{RGB}{117,0,0} 
\definecolor{gnred3}{RGB}{164,0,0} 
\definecolor{gnred4}{RGB}{211,0,0} 
\definecolor{gnred5}{RGB}{255,0,0} 
\definecolor{gnred6}{RGB}{255,42,34} 
\definecolor{gnred7}{RGB}{255,91,89} 

\definecolor{gnblue1}{RGB}{0,36,71}   
\definecolor{gnblue2}{RGB}{0,60,118}  
\definecolor{gnblue3}{RGB}{0,85,164}
\definecolor{gnblue4}{RGB}{0,108,212}
\definecolor{gnblue4}{RGB}{0,108,212}
\definecolor{gnblue5}{RGB}{0,133,255}  
\definecolor{gnblue6}{RGB}{35,156,255} 
\definecolor{gnblue7}{RGB}{88,177,255} 

\definecolor{gnbrown1}{RGB}{71,27,0}  
\definecolor{gnbrown2}{RGB}{117,45,0} 
\definecolor{gnbrown3}{RGB}{164,62,0} 
\definecolor{gnbrown4}{RGB}{211,80,0} 
\definecolor{gnbrown5}{RGB}{255,97,0} 
\definecolor{gnbrown6}{RGB}{255,127,26} 
\definecolor{gnbrown7}{RGB}{255,155,86} 

\newcommand\Item[1][]{%
  \ifx\relax#1\relax  \item \else \item[#1] \fi
  \abovedisplayskip=0pt\abovedisplayshortskip=0pt~\vspace*{-\baselineskip}}

\title{\bf Proximal Gradient Dynamics and Feedback Control for Equality-Constrained Composite Optimization}

\author{
  Veronica Centorrino\thanks{ETH, Zürich {\tt\small vcentorrino@control.ee.ethz.ch}}
\and Francesca Rossi \thanks{Scuola Superiore Meridionale, Italy. {\tt\small f.rossi@ssmeridionale.it}.}
\and Francesco Bullo\thanks{Center for Control, Dynamical Systems, and Computation, UC Santa Barbara, CA, USA. {\tt\small bullo@ucsb.edu}. FB is supported in part by AFOSR grant FA9550-22-1-0059.} 
\and Giovanni Russo\thanks{DIEM, University of Salerno, Italy {\tt\small giovarusso@unisa.it}. GR is supported by the European Union-Next Generation EU Mission 4 Component 1 CUP E53D23014640001.}
}
\begin{document}
\pagestyle{plain}

\maketitle

\begin{abstract}
This paper studies equality-constrained composite minimization problems. This class of problems, capturing regularization terms and inequality constraints, naturally arises in a wide range of engineering and machine learning applications. To tackle these optimization problems, inspired by recent results, we introduce the \emph{proportional--integral proximal gradient dynamics} (PI--PGD): a closed-loop system where the Lagrange multipliers are control inputs and states are the problem decision variables. First, we establish the equivalence between the stationary points of the minimization problem and the equilibria of the PI--PGD. Then for the case of affine constraints, by leveraging tools from contraction theory we give a comprehensive convergence analysis for the dynamics, showing convergence to a stationary point. Moreover, under suitable assumptions, we show linear--exponential convergence towards the equilibrium. That is, the distance between each solution and the equilibrium is upper bounded by a function that first decreases linearly and then exponentially. Our findings are illustrated numerically on a set of representative examples, which include an exploratory application to nonlinear equality constraints.
\end{abstract}

\section{Introduction}
We study equality-constrained non-smooth composite optimization problems (OPs), i.e., problems of the form
\beq
\label{eq:eq_constrained_non_smooth}
\begin{aligned}
\min_{x \in \R^n} \ & f(x) + g(x)\\
\text{ s.t. } & h(x) = \0_m,
\end{aligned}
\eeq
where $\map{f}{\R^n}{\R}$ and $\map{h}{\R^n}{\R^m}$ are differentiable, while $\map{g}{\R^n}{\R}$ is convex, closed, and proper, possibly non-smooth.
These problems arise in many engineering, scientific, and machine learning applications, as they capture regularization terms and convex inequality constraints.
A possible approach to solving~\eqref{eq:eq_constrained_non_smooth} is to use projected dynamical systems or to incorporate the equality constraints into the objective via penalty or augmented Lagrangian methods, and then solve the resulting proximal-gradient dynamics~\cite{NKD-SZK-MRJ:19}.
However, such an approach can become challenging when the proximal operator lacks a closed-form expression or the projection is difficult to compute.

Constrained optimization algorithms can also be interpreted as closed-loop systems, whose goal is to ensure convergence to the optimizer while enforcing feasibility. In this context, in~\cite{VC-SMF-SP-DR:25} a continuous-time control-theoretic framework for equality-constrained smooth optimization has been proposed.
The core idea (see Figure~\ref{fig:control_goal}) is to consider as {\em plant} a dynamical system inspired by the gradient flow with respect to the primal variables of the Lagrangian. The output of this system is $y(t) = h(x(t))$ and the control variables are Lagrange multipliers. The control objective is then to design a feedback controller that drives the output to zero.

\begin{figure}[!h]
    \centering \includegraphics[width=0.48\linewidth]{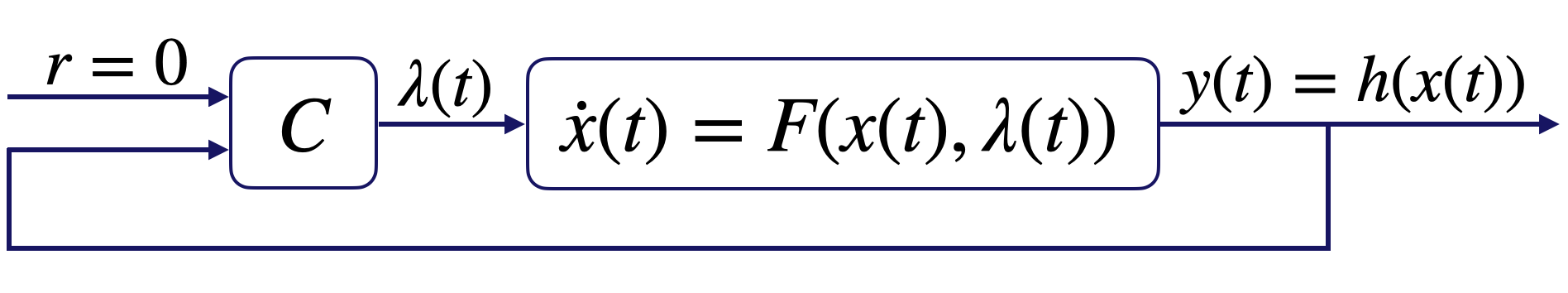}
    \caption{Closed-loop system for equality-constrained OPs: $\dot{x} = F(x, \lambda)$ has the stationary points of the Lagrangian as equilibria. The feedback controller, $C$, is designed to drive $y$ toward the reference value $r = 0$, so that $h(x) = \0_m$.}
    \label{fig:control_goal}
\end{figure}

Inspired by the approach in~\cite{VC-SMF-SP-DR:25}, we propose the \emph{proportional--integral controlled proximal gradient dynamics} (PI--PGD) to solve equality-constrained (non-smooth) composite OPs. For such dynamics, we provide a comprehensive analysis characterizing convergence.

\textit{Literature review:}
Studying OPs as continuous-time dynamics is a classical problem dating back to~\cite{KJA-LH-HU:58}, which has gained renewed interest thanks to developments from, e.g., online  feedback optimization~\cite{GB-JC-JIP-EDA:22} and the analysis of algorithms from a feedback control perspective~\cite{AH-ZH-SB-GH-FD:24}.
The standard approach for constrained OPs is via standard and augmented primal-dual dynamics for smooth cost~\cite{GQ-NL:19} and non-smooth composite OPs~\cite{NKD-SZK-MRJ:19}, respectively. The resulting dynamics are globally exponentially convergent for $f$ strongly convex and $L$-smooth, and $g$ convex, proper and closed.
The use of Lagrange multipliers as feedback controllers has been proposed for smooth OPs with equality~\cite{VC-SMF-SP-DR:25, RZ-AR-JS-NL:25} and inequality~\cite{VC-SMF-SP-DR:24b, RZ-AR-JS-NL:25} constraints, and for smooth constrained OPs via control barrier functions~\cite{AA-JC:24}. Global exponential convergence with PI controller is proved in~\cite{VC-SMF-SP-DR:25, VC-SMF-SP-DR:24b} for the case of full-rank linear constraints and strongly convex and $L$-smooth costs.
More broadly, there has been a growing interest in using strongly contracting dynamics to tackle OPs~\cite{HDN-TLV-KT-JJES:18,VC-AG-AD-GR-FB:23c, AD-VC-AG-GR-FB:23f_corrected} thanks to the highly ordered transient and asymptotic behavior properties enjoyed by such dynamics~\cite{FB:24-CTDS}.
The asymptotic behavior of weakly contracting dynamics has been characterized in, e.g.,~\cite{SC:19} for monotone systems, and~\cite{VC-AD-AG-GR-FB:24a} for convex OPs with a unique minimizer.

\textit{Contributions:}
We propose the PI--PGD for solving equality-constrained composite OPs~\eqref{eq:eq_constrained_non_smooth}. The PI--PGD is a closed-loop system, where: (i) the dynamics for the primal variables (i.e., the plant) has the stationary points of the Lagrangian of~\eqref{eq:eq_constrained_non_smooth} as equilibria; (ii) the dual variables are control inputs and a PI controller is designed so that the closed-loop system converges to an equilibrium, which is a (feasible) minimum for~\eqref{eq:eq_constrained_non_smooth}. We prove the equivalence between the {stationary points} of the OP and the equilibria of the proposed PI--PGD.
Then, we conduct a comprehensive convergence analysis for the widely considered case of linear-equality constraints and strongly convex and $L$-smooth cost term. Our main convergence result shows convergence to a stationary. Moreover, we show that if the equilibrium is locally exponentially stable, convergence towards the equilibrium is \emph{linear-exponential}. That is, the distance between each solution of the dynamics and the equilibrium is upper bounded by a linear--exponential function~\cite{VC-AD-AG-GR-FB:24a}. To establish this convergence property, we leverage contraction theory and characterize global weak infinitesimal contractivity/non-expansiveness and local strong infinitesimal contractivity of the dynamics.
We remark that while we extend the setup in~\cite{VC-SMF-SP-DR:25} by including a non-smooth term, our main contribution lies in the convergence analysis tailored to handle the non-smooth term, which requires a distinct mathematical approach, interesting \emph{per se}.
%
Finally, we validate our results via numerical examples, including an exploratory application to nonlinear equality constraints. Notably, the proposed PI--PGD still converges to the optimal solution beyond the theoretical assumptions.
The code to replicate our numerical results is available at~\url{https://shorturl.at/mPpEZ}.

\section{Mathematical Preliminaries}
We denote by $\0_n$ and $\1_n \in \R^n$ the all-zeros and all-ones vector of size $n$, respectively. Vector inequalities of the form $x \leq (\geq) y$ are entrywise. We let $I_n$ be the $n \times n$ identity matrix. 
The symbol $\otimes$ denotes the Kronecker product.
Given $A, B \in \R^{n\times n}$ symmetric, we write $A \preceq B$ (resp. $A \prec B$) if $B-A$ is positive semidefinite (resp. definite). When $A$ has only real eigenvalues, we let $\subscr{\lambda}{min}(A)$ and $\subscr{\lambda}{max}(A)$ be its minimun and maximum eigenvalue, respectively.
We say that $A$ is \emph{Hurwitz} if $\alpha(A) := \max \setdef{\operatorname{Re}(\lambda)}{\lambda \text{ eigenvalue of } A}<0$, where $\operatorname{Re}(\lambda)$ denotes the real part of $\lambda$.  {Given $x \in \R^n$, we let $\diag{x} \in \R^{n \times n}$ be the diagonal matrix with diagonal entries equal to $x$}.

\paragraph{Norms, Logarithmic Norms and Weak Pairings.}
We let $\| \cdot \|$ denote both a norm on $\R^n$ and its corresponding induced matrix norm on $\R^{n \times n}$. Given $x \in \R^n$ and $\radius >0$, we let $\ball{p}{x,\radius}:= \setdef{z\in\R^n}{\norm{z-x}_p\leq\radius}$ be the \emph{ball of radius $\radius$ centered at $x$} computed with respect to the norm $p$.

Given $A \in \R^{n \times n}$ the \emph{logarithmic norm} (lognorm) induced by $\| \cdot \|$ is 
$$\mu(A) := \lim_{h\to 0^+} \frac{\norm{I_n{+}h A} -1}{h}.$$
Given a symmetric positive-definite matrix $P \in \R^{n\times n}$, we let $\|\cdot\|_{P}$ be the $P$-weighted $\ell_2$ norm $\|x\|_{P} := \sqrt{x^\top P x}$, $x \in \R^n$ and write $\|\cdot\|_2$ if $P = I_n$. The corresponding lognorm is $\mu_{P}(A) =\min\setdef{b \in \R}{PA + A^\top P \preceq 2bP}$~\cite[Lemma~2.7]{FB:24-CTDS}.

\paragraph{Convex Analysis and Proximal Operators.}
Given a convex set $\mcC$, the function $\map{\iota_{\mcC}}{\R^n}{[0,+\infty]}$ is the \emph{zero-infinity indicator function on $\mcC$} and is defined by $\iota_{\mcC}(x) = 0$ if $x \in \mcC$ and $\iota_{\mcC}(x) = +\infty$ otherwise.  The function $\map{\relu}{\R}{\R_{\geq 0}}$, is defined by $\relu(x) = \max\{0, x\}$. A map $\map{g}{\R^n}{\realextended}$ is (i) \emph{convex} if $\operatorname{epi}(g) := \setdef{(x,y) \in \R^{n+1}}{g(x) \leq y}$ is a convex set; (ii) \emph{proper} if its value is never $-\infty$ and there exists at least one $x \in \R^n$ such that $g(x) < \infty$; (iii) \emph{closed} if it is proper and $\operatorname{epi}(g)$ is closed. {We denote by $\partial g$ the \emph{subdifferential} of $g$.}
A map $\map{g}{\R^n}{\realextended}$ is 
\begin{enumerate}[label = (\roman*)]
\item \emph{strongly convex with parameter $\rho > 0$} if the map $x \mapsto g(x) - \frac{\mf}{2}\|x\|_2^2$ is convex;
\item { \emph{$L$-smooth} if it is differentiable and $\nabla g$ is Lipschitz with constant $L >0$.}
\end{enumerate}

Next, we define the proximal operator of $g$, which is a map that takes a vector $x \in \R^n$ and maps it into a subset of $\R^n$, which can be either empty, contain a single element, or be a set with multiple vectors.
\bd[Proximal Operator]
\label{apx:def:prox_operator}
The \emph{proximal operator} of a function $\map{g}{\R^n}{\realextended}$ with parameter $\gamma>0$, $\map{\prox{\gamma g}}{\R^n}{\R^n}$, is the operator given by
\beq
 \prox{\gamma g}(x) = \displaystyle \argmin_{z \in \R^n} g(z) + \frac{1}{2 \gamma}\|x - z\|_2^2, \quad \forall x \in \R^n.
\eeq
\ed
The map $\prox{\gamma g}$ is firmly nonexpansive~\cite[Proposition~12.28]{HHB-PLC:17}.
The subdifferential operator of $g$, $\partial g$, is linked to the proximal operator $\prox{\gamma g}$ by the following relation
\beq
\prox{\gamma g} = (I_n + \gamma \partial g)^{-1}.
\eeq
The (point-to-point) map $(I_n + \gamma \partial g)^{-1}$ is called the \emph{resolvent} of the operator $\partial g$ with parameter $\gamma > 0$. That is, the proximal operator is the resolvent of the subdifferential operator. Moreover, when the function $g$ is closed, convex, and proper, then the resolvent, and so the proximal map, is single-valued, even though $\partial g$ is not.
{We denote by $D(h(\cdot))$ the Jacobian of a map $\map{h}{\R^n}{\R^m}$. We conclude recalling a well-known result on first-order necessary conditions for optimality (see, e.g.,~\cite{RTR:70}).
\bt[First-order necessary conditions]
\label{thm:FONC}
Consider problem~\eqref{eq:eq_constrained_non_smooth} and let $x^{\star} \in \R^{n}$ be a local minimum satisfying $h(x^{\star}) = \0_m$. Assume that $x^{\star}$ is regular, that is, the rows of $Dh(x^\star)$ are linearly independent. Then, there exists $\lambda^{\star} \in \R^{m}$ such that $(x^{\star}, \lambda^{\star})$ is a saddle point of the Lagrangian, that is,
$$
\0_n \in \nabla f(x^{\star}) + \partial g(x^{\star}) + D h(x^{\star})^{\top} \lambda^{\star}.
$$
\et
The regularity assumption in Theorem~\ref{thm:FONC}, namely that the rows of $D h(x^{\star})$ are linearly independent, is also known as the \emph{Linear Independence Constraint Qualification} (LICQ). Finally, we recall that a stationary point $(\bar x, \bar \lambda)$ of~\eqref{eq:eq_constrained_non_smooth} satisfies the first-order necessary conditions but is not necessarily a local or global minimizer unless additional assumptions hold (e.g., convexity of $f$ and $g$, and affine constraints).  
Conversely, any local minimizer of~\eqref{eq:eq_constrained_non_smooth} satisfying, for example, the LICQ, is a stationary point.
}
\subsection{Contraction Theory for Dynamical Systems.}
Consider a dynamical system 
\beq
\label{eq:dynamical_system}
\dot{x}(t) = f\bigl(t,x(t)\bigr),
\eeq 
where $\map{f}{\R_{\geq 0} \times \mcC}{\R^n}$, is a smooth nonlinear function with $\mcC\subseteq \R^n$ forward invariant set for the dynamics. We let $t \mapsto \odeflowtx{t}{x_0}$ be the flow map of~\eqref{eq:dynamical_system} at time $t$ starting from initial condition $x(0):= x_0$.
Then, we give the following~\cite{GR-MDB-EDS:10a, FB:24-CTDS}:
\bd[Contracting dynamics] \label{def:contracting_system}
Given a norm $\norm{\cdot}$ with associated lognorm $\mu$, a smooth function $\map{f}{\R_{\geq 0} \times \mcC}{\R^n}$, with $\mcC \subseteq \R^n$ $f$-invariant, open and convex, and a \emph{contraction rate} $c >0$ ($c = 0)$, $f$ is $c$-strongly (weakly) infinitesimally contracting on $\mcC$ if
\beq
\label{cond:contraction_log_norm}
\mu\bigl(Df(t, x)\bigr) \leq -c,
\eeq
for all $x \in \mcC  \textup{ and } t\in \R_{\geq0}$, where $Df(t,x) := \partial f(t,x)/\partial x$.
\ed

If $f$ is contracting, then for any two trajectories $x(\cdot)$ and $y(\cdot)$ of~\eqref{eq:dynamical_system} it holds that
$$\|\odeflowtx{t}{x_0} - \phi_t(y_0)\| \leq \e^{-ct}\|x_0 -y_0\|, \quad \textup{ for all } t \geq 0,$$
i.e., the distance between the two trajectories converges exponentially with rate $c$ if $f$ is $c$-strongly infinitesimally contracting, and never increases if $f$ is weakly infinitesimally contracting.

In~\cite[Theorem 16]{AD-AVP-FB:22q} condition~\eqref{cond:contraction_log_norm} is generalized for locally Lipschitz function, for which, by Rademacher’s theorem, the Jacobian exists almost everywhere (a.e.) in $\mcC$. Specifically, if $f$ is locally Lipschitz, then  $f$ is infinitesimally contracting on $\mcC$ if condition~\eqref{cond:contraction_log_norm} holds for almost every $x \in \mcC$ and $t\in \R_{\geq0}$.

Finally, we recall the following result on the convergence behavior of globally-weakly and locally-strongly contracting dynamics. We refer to~\cite{VC-AD-AG-GR-FB:24a} for more details.

\bt[Linear-exponential convergence of globally-weakly and locally-strongly contracting dynamics]
\label{thm:local-exp-contractivity_same_norms}
Consider the dynamics~\eqref{eq:dynamical_system}. Assume that: (i) $f$ is weakly infinitesimally contracting on $\R^n$ w.r.t. $\norm{\cdot}_{\glo}$; (ii) strongly infinitesimally contracting {with rate $\ce$} on a forward-invariant set $\mcS$ w.r.t. $\norm{\cdot}_{\loc}$; (iii) there exists an equilibrium point $\xstar \in \mcS$. Let $\radius$ be the largest radius satisfying $\ball{\loc}{\xstar, r} \subseteq \mcS$. For each trajectory $x(t)$ starting from $x_0$ it holds that
\begin{enumerate}
\item
\label{item:1_equal_norm}
if $x_0 \in \mcS$, then, for almost every $t \geq 0$,
\[
\norm{x(t) - \xstar} \leq \e^{-\ce t} \norm{x_0 - \xstar};
\]
\item
\label{item:2_equal_norm}
if $x_0 \notin \mcS$, then, for almost every $t \geq 0$,
\beq
\label{eq:bound_outside_the_ball}
\norm{x(t) - \xstar} \leq \linexp{t  ; q, \cl, \ce, \tld} :=
\begin{cases}
q - \cl t  \ & \textup{ if } t \leq \tld,\\
\bigl(q - \cl \tld\bigr) \e^{ - \ce (t - \tld)}\  & \textup{ if }  t > \tld,\\
\end{cases}
\eeq
where {the parameters $\cl$, $q$, and $\tld$ are given in~\cite{VC-AD-AG-GR-FB:24a}.}
\end{enumerate}
\et
For brevity, we say that every solution of~\eqref{eq:dynamical_system} satisfying Theorem~\ref{thm:local-exp-contractivity_same_norms} \emph{linear-exponentially converge} towards its equilibrium {with respect to the norms $\norm{\cdot}_{\glo}$ and $\norm{\cdot}_{\loc}$}.

\section{Equality-Constrained Composite Optimization via Feedback Control}
Consider the equality-constrained composite OP~\eqref{eq:eq_constrained_non_smooth} that we rewrite here for convenience
\begin{equation*}
\begin{aligned}
\min_{x \in \R^n} \ & f(x) + g(x)\\
\text{ s.t. } & h(x) = \0_m,
\end{aligned}
\end{equation*}
where $\map{f}{\R^n}{\R}$ and $\map{h}{\R^n}{\R^m}$ are differentiable functions, while $\map{g}{\R^n}{\R}$ is a convex, closed, and proper function, possibly non-smooth. Problem~\eqref{eq:eq_constrained_non_smooth} includes inequality-constrained optimization problems. To see this, consider the problem
\begin{align*}
\min_{x \in \R^n} \ & f(x)\\
\text{ s.t. } & g(x) \leq 0 \\ 
& h(x) = \0_m,
\end{align*}
and note that the above minimization problem can be equivalently rewritten as
\begin{align*}
\min_{x \in \R^n} \ & f(x) + \iota_{\setdef{x}{g(x) \leq 0}}(x)\\
\text{ s.t. } & h(x) = \0_m.
\end{align*}

Inspired by the approach in~\cite{VC-SMF-SP-DR:25}, to solve problem~\eqref{eq:eq_constrained_non_smooth}, we propose a continuous-time closed-loop dynamical system, for which we design a suitable feedback controller  driving the dynamics towards a minimizer of~\eqref{eq:eq_constrained_non_smooth}. To this end, consider the Lagrangian associated with the minimization problem~\eqref{eq:eq_constrained_non_smooth}, that is the map $\map{L}{\R^n \times \R^m}{\R}$
\begin{equation*}
L(x, \lambda)=f(x)+ g(x) + \lambda^{\top}h(x),
\end{equation*}
where $\lambda \in \R^m$ is the vector of Lagrange multipliers {associated with the equality constraint $h(x) = \0_m$}. 
\bd[Stationary point of~\eqref{eq:eq_constrained_non_smooth}]
Consider the equality-constrained OP~\eqref{eq:eq_constrained_non_smooth}. A \emph{stationary point for problem~\eqref{eq:eq_constrained_non_smooth}} is any pair $(x^{\star},\lambda^{\star}) \in \R^{n + m}$ such that $\0_n \in \nabla f(x^{\star}) + \partial g(\xstar) +  D\bigl(h(x^{\star})\bigr)^{\top}\lambda^{\star}$, and  $h\left(x^{\star}\right)= \0_m$.
\ed

\begin{rem}[{First-order necessary and sufficient conditions for convex OP}]
\label{rem1}
Assume that Problem~\eqref{eq:eq_constrained_non_smooth} is convex, that is, $f$ is convex and the constraints are affine, i.e.,  $h(x) = Ax - b$, with $A \in \R^{m \times n}$ and $b \in \R^m$.
By the first-order necessary and sufficient optimality conditions for convex optimization problems~\cite{RTR:70}, a vector $x^\star \in \R^n$ satisfying $A x^\star - b = \0_m$ is a minimizer of~\eqref{eq:eq_constrained_non_smooth} if and only if there exists $\lambda^\star \in \R^m$ such that $(x^\star, \lambda^\star)$ is a saddle point of the Lagrangian, that is $\0_n \in \nabla f(x^{\star}) + \partial g(\xstar) +  A^{\top}\lambda^{\star}$, and $A\xstar - b = \0_m$.
Therefore, if $(x^\star, \lambda^\star) \in \R^{n+m}$ is a stationary point of~\eqref{eq:eq_constrained_non_smooth}, then $x^\star$ is a global minimizer of the problem.
\end{rem}

Interpreting the Lagrange multipliers $\lambda(t) \in \R^m$ as a control input, we consider the following proximal forward–backward dynamics
\beq
\label{eq:system_prox_FB_controlled}
\begin{cases}
\dot{x}(t)= - x(t) + \prox{\gamma g}\bigl(x(t) - \gamma(\nabla f(x(t)) + D\bigl(h(x(t))\bigr)^{\top}\lambda(t))\bigr)\\
y(t)=h(x(t)),
\end{cases}
\eeq
where $x \in \R^n$ is the state, $y \in \R^m$ is the output and $\gamma >0$ is a parameter.
This parameter regulates the influence of $\nabla f$ and the feedback term $ D\bigl(h(x)\bigr)^{\top}\lambda$ within the proximal operator and affects the convergence properties (see Section~\ref{sec:convergence}).

We assume that Problem~\eqref{eq:eq_constrained_non_smooth} admits at least one feasible point, that is, there exists $\bar x$ such that $h(\bar x)=\0_m$.
The following result establishes the connection between the stationary point of problem~\eqref{eq:eq_constrained_non_smooth} and the equilibrium points of the dynamics~\eqref{eq:system_prox_FB_controlled}.

\begin{lem}[Linking the stationary point of~\eqref{eq:eq_constrained_non_smooth} and the equilibria of~\eqref{eq:system_prox_FB_controlled}]
\label{lem:equivalence}
A point $\left(x^\star, \lambda^\star\right) \in \R^{n+m}$ is a stationary point of~\eqref{eq:eq_constrained_non_smooth} if and only if is an equilibrium point of system~\eqref{eq:system_prox_FB_controlled} with input $\lambda^\star$, satisfying $h(x^\star) = \0_m$. 
\end{lem}
\begin{proof}
Let $\left(x^\star, \lambda^\star\right) \in \R^{n+m}$ be a stationary point of~\eqref{eq:eq_constrained_non_smooth}. Then $\0_n \in \nabla f(x^{\star}) + \partial g(\xstar) +  D\bigl(h(x^{\star})\bigr)^{\top}\lambda^{\star}$.
Multiplying both sides by $\gamma >0$ and adding and subtracting $x^\star$ to the right-hand side of the above inclusion yields
\[
\begin{aligned}
\0_n \in [I_n + \gamma\partial g](\xstar) {+} \gamma \nabla f(\xstar) {+} \gamma D\bigl(h(x^{\star})\bigr)^{\top}\lambda^{\star} {-} \xstar &\iff (I_n + \gamma\partial g)(\xstar) \ {\ni} \ \xstar - \gamma \bigl(\nabla f(\xstar) + D\bigl(h(x^{\star})\bigr)^{\top}\lambda^{\star}\bigr)\\
&\iff \xstar \in (I_n + \gamma\partial g)^{-1}\left(\xstar - \gamma \bigl(\nabla f(\xstar) + D\bigl(h(x^{\star})\bigr)^{\top}\lambda^{\star}\bigr)\right).
\end{aligned}
\]
Recalling that $\prox{\gamma g} = (I_n + \gamma\partial g)^{-1}$ and, being by assumption $g$ convex, closed, and proper, then $\prox{\gamma g}$ is single-valued~\cite{NP-SB:14}. Therefore, we have
$$
\xstar = \prox{\gamma g}\bigl(\xstar - \gamma(\nabla f(\xstar) + D\bigl(h(x^{\star})\bigr)^{\top}\lambda^{\star})\bigr).
$$
That is, $\xstar$ is an equilibrium point of problem~\eqref{eq:system_prox_FB_controlled}. Specifically, $\xstar$ is the equilibrium with input $\lambda^\star$.
Since all steps are equivalence, the proof is complete.
\end{proof}
To compute the minimizers of~\eqref{eq:eq_constrained_non_smooth}, we design a control input $\lambda(t)$ driving~\eqref{eq:system_prox_FB_controlled} to an equilibrium $x^\star$, which is a stationary point of~\eqref{eq:eq_constrained_non_smooth}, and regulating the output to zero, ensuring $x^\star$ is feasible. Specifically, $\lambda(t)$ is the output of a PI controller, so that:
\begin{equation*}
    \lambda(t)=\subscr{k}{p} y(t) + \subscr{k}{i} \int_{0}^{t} y(\tau) d\tau,
\end{equation*}
where $\subscr{k}{p}, \subscr{k}{i} \in \R_{>0}$ are the control gains. Differentiating, we have
\beq
\label{eq:pi_controller}
\dot \lambda(t) = \subscr{k}{p} \dot y(t) + \subscr{k}{i} y(t).
\eeq
Then, the closed-loop dynamics composed by system~\eqref{eq:system_prox_FB_controlled} and the \PI controller~\eqref{eq:pi_controller} -- illustrated in Figure~\ref{fig:pi-pgd_affine} -- is the following continuous-time \emph{PI proximal-gradient dynamics} (PI--PGD)
\[
\begin{cases}
\dot{x} = - x + \prox{\gamma g}\bigl(x - \gamma(\nabla f(x) + D\bigl(h(x)\bigr)^{\top}\lambda)\bigr) \\
\dot \lambda = \subscr{k}{p} D\bigl(h(x)\bigr) \dot x + \subscr{k}{i} h(x),
\end{cases}
\]
or equivalently
\beq
\label{eq:pi-pgd}
\begin{cases}
\dot{x} & = - x + \prox{\gamma g}\bigl(x - \gamma(\nabla f(x) + D\bigl(h(x)\bigr)^{\top}\lambda)\bigr) \\
\dot \lambda & = \subscr{k}{p} D\bigl(h(x)\bigr) \bigl(- x + \prox{\gamma g}\bigl(x - \gamma(\nabla f(x) + D\bigl(h(x)\bigr)^{\top}\lambda)\bigr) \bigr)+ \subscr{k}{i} h(x).
\end{cases}
\eeq
\begin{figure}[!h]
\centering
\includegraphics[width=.85\linewidth]{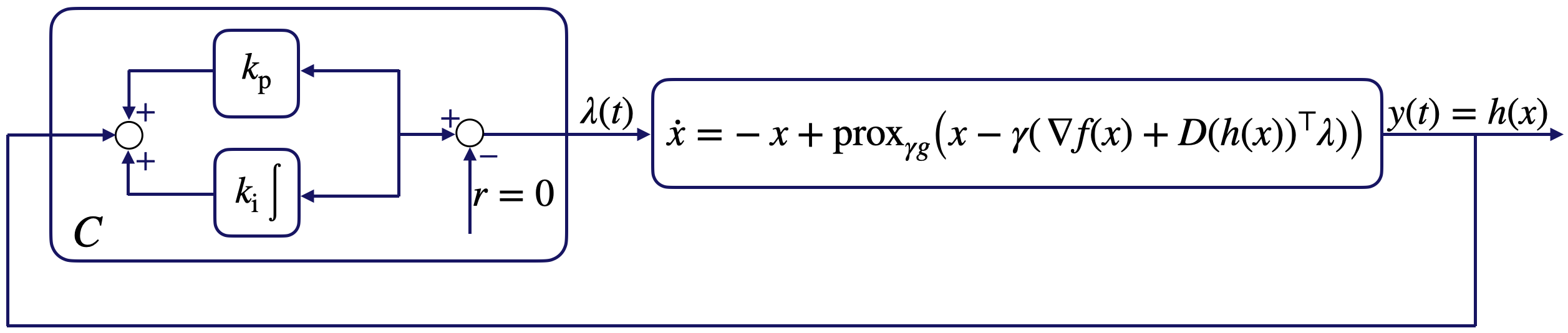}
\caption{PI--PGD: Closed-loop dynamics composed by system~\eqref{eq:system_prox_FB_controlled} and the \PI controller~\eqref{eq:pi_controller}.}
\label{fig:pi-pgd_affine}
\end{figure}

The following result is an immediate consequence of Lemma~\ref{lem:equivalence}.
\begin{cor}
\label{lem:equivalence_dynamics}
Assume that Problem~\eqref{eq:eq_constrained_non_smooth} {is convex}. Then $(\xstar, \lambda^\star) \in \R^n \times \R^m$ is an equilibrium point of~\eqref{eq:pi-pgd} if and only if $\xstar$ is a minimizer of~\eqref{eq:eq_constrained_non_smooth}.
\end{cor}
\begin{proof}
The claim follows directly from Remark~\ref{rem1} and Lemma~\ref{lem:equivalence}, noting that if $(\xstar, \lambda^\star)$ is an equilibrium point of~\eqref{eq:pi-pgd}, then $h(\xstar) = \0_m$. 
\end{proof}

\begin{rem}[Handling nonlinear inequality constraints via slack variables]
Consider an optimization problem of the form
\[
\begin{aligned}
\min_{x \in \R^n} \ & f(x) + g(x)\\
\text{ s.t. } & h(x) = \0_m, \\
& q(x) \leq \0_r,
\end{aligned}
\]
where $f$ and $h$ satisfy the same assumptions as in Problem~\eqref{eq:eq_constrained_non_smooth}, and $q : \R^n \to \R^r$ is possibly nonlinear. The problem above can be equivalently rewritten in the form of Problem~\eqref{eq:eq_constrained_non_smooth}  by introducing slack variables $s \in \R^r_{\ge 0}$:
\[
\begin{aligned}
\min_{x \in \R^n, s \in \R^r} \ & f(x) + g(x) + \iota_{\R^r_{\geq 0}}(s)\\
\text{ s.t. } & {\tilde h}(x, s) = \0_{m + r},
\end{aligned}
\]
where $\tilde{h}(x,s) := \begin{bmatrix} h(x) \\ q(x) + s \end{bmatrix}$. The corresponding PI--PGD dynamics for the augmented variable $(x,s)$ are
\[
\begin{cases}
\dot{x} = -x + \prox{\gamma g}\big(x - \gamma (\nabla f(x) + D h(x)^\top \lambda_h + D q(x)^\top \lambda_q)\big),\\[1mm]
\dot{s} = -s + \relu\big(s - \gamma \lambda_q\big),\\[1mm]
\dot{\lambda}_h = k_p D h(x) \dot{x} + k_i h(x),\\[1mm]
\dot{\lambda}_q = k_p (D q(x) \dot{x} + \dot{s}) + k_i (q(x) + s),
\end{cases}
\]
where $(\lambda_h, \lambda_q)$ denote the dual variables associated with the equality and transformed inequality constraints, respectively.
Such structure-preserving reformulations are particularly useful when the feasible set is nonconvex or lacks a closed-form projection.
\end{rem}

\section{Convergence of the PI Proximal-Gradient Dynamics}
\label{sec:convergence}
We now focus on the case of affine constraints and study the convergence properties of the resulting PI--PGD. That is, we let $h(x) = Ax - b$, where $A \in \R^{m \times n}$, $b \in \R^m$.
The PI--PGD then becomes
\beq
\label{eq:pi-pgd_affine}
\begin{cases}
\dot{x} = - x + \prox{\gamma g}\left(x - \gamma(\nabla f(x) + A^{\top}\lambda)\right) \\
\dot \lambda = (\subscr{k}{i} - \subscr{k}{p})Ax + \subscr{k}{p}A\prox{\gamma g}\left(x - \gamma(\nabla f(x) + A^{\top}\lambda)\right) - \subscr{k}{i}b.
\end{cases}
\eeq
In what follows, we let $z = (x,\lambda) \in \R^{n + m}$ and let $\map{\subscr{F}{PGD}}{\R^{n+m}}{\R^{n+m}}$ be the vector field~\eqref{eq:pi-pgd_affine} for $\dot{z} = \subscr{F}{PGD}(z)$.
Additionally, given a point $z^\star \in \R^{n+m}$, we let $\subscr{\Omega}{F}(z^\star)$ be the set of differentiable points of $\subscr{F}{PGD}$ in a neighborhood of $z^\star$.

We make the following standard assumptions on the function $f$ and the matrix $A$.
\begin{assumption*}
For the dynamics~\eqref{eq:pi-pgd_affine}, assume
\begin{enumerate}[label=\textup{($A$\arabic*)}]
\item 
\label{ass:1_f}
the function $\map{f}{\R^n}{\R}$ is strongly convex with parameter $\rho$ {and $L$-smooth};
\item
\label{ass:2_A}
the matrix $A \in \R^{m \times n}$ has full row rank and satisfies $\amin I_m \preceq AA^\top \preceq \amax I_m$ for $\amin,\amax \in \R_{>0}$.
\end{enumerate}
\end{assumption*}
Under Assumptions~\ref{ass:1_f} and~\ref{ass:2_A}, the OP~\eqref{eq:eq_constrained_non_smooth} has a unique solution. As a consequence, these assumptions are standard in the literature when establishing global convergence to the equilibrium (see, e.g.,~\cite{GQ-NL:19, NKD-SZK-MRJ:19, AD-VC-AG-GR-FB:23f_corrected, VC-SMF-SP-DR:25}).
Regarding the set of associated Lagrange multipliers $\Lambda^\star$, Assumption~\ref{ass:2_A}, together with the convexity assumptions and feasibility, guarantees that this set is nonempty. Moreover, $\Lambda^\star$ is convex. However, it is generally not a singleton when $g$ is non-smooth. Indeed, at optimality, the stationary condition reads $A^\top\lambda^\star \in -\nabla f(x^\star) - \partial g(x^\star)$. By Assumption~\ref{ass:2_A}, $A$ has full row rank, implying that $A^\top$ is injective. Then, $\lambda^\star$ is unique if and only if the affine subspace $-\nabla f(x^\star) + \mathcal{R}(A^\top)$ intersects the subdifferential $\partial g(x^\star)$ at exactly one point, where $\mathcal{R}(A^\top)$ denotes the column space of $A^\top$. A sufficient condition for this is, e.g., smoothness of $g$ at $x^\star$, which reduces $\partial g(x^\star)$ to a singleton and ensures a unique intersection. More generally, the uniqueness of $\lambda^\star$ can also be established for certain classes of non-smooth functions $g$ under suitable strict complementarity conditions~\cite{FHC:83}.

We begin our analysis by showing that, for proper parameter choice, the PI--PGD is weakly infinitesimally contracting, which in turn implies that the distance between any two trajectories of the PI--PGD never increases.
\begin{lem}[Global weak contractivity of~\eqref{eq:pi-pgd_affine}]
\label{thm:weak-contractivity}
Consider the PI--PGD~\eqref{eq:pi-pgd_affine} with $\subscr{k}{p} = \subscr{k}{i} >0$ satisfying Assumption~\ref{ass:1_f}. For any $p \in \left[\max\left(\frac{\subscr{k}{p}L}{3}, \frac{\subscr{k}{p}(1 - 2\gamma \rho)}{\gamma}\right), \frac{\subscr{k}{p}}{\gamma}\right]$ and for any $\gamma \in \left]0, \frac{1}{L}\right]$, the PI--PGD~\eqref{eq:pi-pgd_affine} is weakly infinitesimally contracting on $\R^{n+m}$ with respect to the norm $\norm{\cdot}_{P}$, where 
\beq
\label{eq:matrix_P}
{P =}
\begin{bmatrix}
p I_n & 0\\
0 & I_m
\end{bmatrix}.
\eeq
\end{lem}
\begin{proof}
For simplicity of notation, let $y := x - \gamma(\nabla f(x) + A^\top \lambda)$, $G(y) := D\prox{\gamma g}(y)$ {and $B(x) := \Hess f(x)$}. Recall that {(i)} $G(y)$ is symmetric and $\prox{\gamma g}$ is nonexpansive {(see, e.g.,~\cite[Lemma 19]{AD-VC-AG-GR-FB:23f_corrected})}, thus $0 \preceq G(y) \preceq I_n$, for a.e. $y$; {(ii) Assumption~\ref{ass:1_f} implies $\rho I_n \preceq B(x) \preceq L I_n$ for all $x \in \R^n$}.
The Jacobian of $\subscr{F}{PGD}$ is
$$
D{\subscr{F}{PGD}}(z) = 
\begin{bmatrix}
- I_n + G(y)(I_n - \gamma B(x))
& -\gamma G(y) A^\top \\
(\subscr{k}{i} - \subscr{k}{p})A + \subscr{k}{p}AG(y)\bigl(I_n - \gamma B(x)\bigr) & -\gamma \subscr{k}{p}AG(y)A^\top
\end{bmatrix},
$$
which exists for almost every $z$. To prove our statement, we have to show that $\mu_{P}(D \subscr{F}{PGD}(z)) \leq 0$, for a.e. $z$. Note that
\begin{multline*}
{
\sup_{z} \mu_{P}(D \subscr{F}{PGD}(z)) \leq \max_{\substack{0 \preceq G \preceq I_n \\ \rho I_n \preceq B \preceq L I_n}}\mu_{P}
\Biggl(
\underbrace{
\begin{bmatrix}
- I_n + G(I_n - \gamma B)
& -\gamma G A^\top \\
(\subscr{k}{i} - \subscr{k}{p})A + \subscr{k}{p}AG\bigl(I_n - \gamma B\bigr) & -\gamma \subscr{k}{p}AGA^\top
\end{bmatrix}
}_{:= J}
\Biggr),
}
\end{multline*}
where the $\sup$ is over all $z$ for which $D\subscr{F}{PGD}(z)$ exists. Then, to prove our statement it suffices to show that the LMI  $ -J^\top P - P J \succeq 0$ is satisfied. We compute
\begin{align*}
- J^\top P &- PJ
=
\begin{bmatrix}
2p I_n - 2p G + p \gamma(BG + GB) 
& (\subscr{k}{p} - \subscr{k}{i})A^\top + \bigl(\gamma p I_n - \subscr{k}{p}(I_n - \gamma B)\bigr) G A^\top \\
(\subscr{k}{p} - \subscr{k}{i})A + AG\bigl(\gamma p I_n - \subscr{k}{p}(I_n - \gamma B)\bigr)
& 2 \gamma \subscr{k}{p}AGA^\top 
\end{bmatrix} \\
&=
\begin{bmatrix}
I_n & 0 \\
0 & A
\end{bmatrix}
\begin{bmatrix}
2p I_n - 2p G + p \gamma(BG + GB) 
& (\subscr{k}{p} - \subscr{k}{i})I_n + \bigl(( \gamma p - \subscr{k}{p}) I_n + \gamma \subscr{k}{p} B\bigr) G\\
(\subscr{k}{p} - \subscr{k}{i})I_n +  G\bigl( (\gamma p - \subscr{k}{p}) I_n + \gamma \subscr{k}{p} B\bigr)
& 2 \gamma \subscr{k}{p}G 
\end{bmatrix}
\begin{bmatrix}
I_n & 0 \\
0 & A^\top
\end{bmatrix}\\
&:=
\begin{bmatrix}
I_n & 0 \\
0 & A
\end{bmatrix}
\underbrace{
\begin{bmatrix}
Q_{11} & Q_{12}\\
Q_{12}^\top &Q_{22}
\end{bmatrix}
}_{:= Q}
\begin{bmatrix}
I_n & 0 \\
0 & A^\top
\end{bmatrix}
\succeq 0 \quad \iff \quad Q \succeq 0,
\end{align*}
where we have introduced the matrix $Q$ with components
\begin{align*}
Q_{11} &= 2p I_n - 2p G + p \gamma(BG + GB);\\
Q_{12} &= (\subscr{k}{p} - \subscr{k}{i})I_n + \bigl(( \gamma p - \subscr{k}{p}) I_n + \gamma \subscr{k}{p} B\bigr) G;\\
Q_{22} &= 2 \gamma \subscr{k}{p}G.
\end{align*}
To show $Q \succeq 0$, we need to prove that (i) $Q_{11} \succ 0$ and (ii) the Schur complement of the block $Q_{11}$ is positive semidefinite, that is, $Q_{22} - Q_{12}^\top Q_{11}^{-1}Q_{12} \succeq 0$.
\paragraph{(i) $\mathbf{Q_{11} \succ 0}$.}
We begin by noting that $G$ is symmetric and satisfies $0 \preceq G \preceq I_n$, then there exists $U$ satisfying $ U U^\top= U^\top U=I_n$ and $g \in [0,1]^n$ such that $G = U\diag{g}U^\top$.
Substituting this equality into $Q_{11}$ and multiplying on the left and on the right by $U^\top$ and $U$, respectively, we get
\begin{align*}
U^\top Q_{11} U &= 2p I_n - 2p \diag{g} + p\gamma(U^\top B U\diag{g}  + \diag{g}U^\top BU) := 2 p\bigl(I_n - \diag{g}\bigr) + p\gamma(X\diag{g}  + \diag{g}X),
\end{align*}
where in the last equality we defined the matrix $X := U^\top B U$. Note that $X = X^\top$, $0 \prec \rho I_n \preceq B \preceq L I_n$ and, by assumption, $\gamma \leq \frac{1}{L}$. Then, by applying Lemma~\ref{lemma:BoundQ11}, for all $g \in [0,1]^n$ it holds
\begin{align}
\label{ineq:q11}
p \gamma(X\diag{g}  + \diag{g}X) + 2p\bigl(I_n - \diag{g}\bigr) \succ \frac{3}{2}p \gamma X.
\end{align}
By multiplying~\eqref{ineq:q11} on the left and on the right by $U$ and $U^\top$, respectively, we get the following lower bound on $Q_{11}$
\begin{align*}
Q_{11} \succ \frac{3}{2}\gamma p B \succeq \frac{3}{2}\gamma \rho p I_n \succ 0.
\end{align*}
{which in turn implies the bound $\frac{3\gamma p}{2}Q_{11}^{-1} \prec B^{-1}$.}
\paragraph{(ii) $\mathbf{Q_{22} - Q_{12}^\top Q_{11}^{-1}Q_{12} \succeq 0}$.} Let 
$Q_{12} = (\subscr{k}{p} - \subscr{k}{i})I_n + \gamma p G - \subscr{k}{p}G + \gamma \subscr{k}{p} BG := B R_1 + R_2
$,
where to simplify the notation we have introduced the matrices $R_1 := \gamma { \subscr{k}{p}} G = R_1^\top$, and  $R_2 := (\subscr{k}{p} - \subscr{k}{i})I_n + (\gamma p - \subscr{k}{p})G = R_2 ^ \top$. Then 
\begin{align*}
\frac{3 \gamma p}{2}Q_{12}^\top Q_{11}^{-1}Q_{12} &\preceq (B R_1 + R_2)^\top B^{-1} (B R_1 + R_2)  \preceq (R_1 B + R_2)B^{-1} (B R_1 + R_2) \\
& = R_1 B R_1 + 2R_1 R_2 + R_2 B^{-1} R_2,
\end{align*}
where 
\begin{align*}
R_1 B R_1 &= \gamma^2 \subscr{k}{p}^2 G B G \preceq \gamma^2 \subscr{k}{p}^2 L G^2\\
R_1 R_2 &= \gamma \subscr{k}{p}(\subscr{k}{p} - \subscr{k}{i}) G + \gamma \subscr{k}{p}(\gamma p - \subscr{k}{p}) G^2 \\
R_2 B^{-1} R_2 &= (\subscr{k}{p} - \subscr{k}{i})^2B^{-1} + (\gamma p - \subscr{k}{p})^2GB^{-1}G + (\subscr{k}{p} - \subscr{k}{i})(\gamma p - \subscr{k}{p}) B^{-1}G + (\subscr{k}{p} - \subscr{k}{i})(\gamma p - \subscr{k}{p})GB^{-1}\\
&\preceq (\subscr{k}{p} - \subscr{k}{i})^2 \rho^{-1}I_n + (\gamma p - \subscr{k}{p})^2\rho^{-1}G^2 + (\subscr{k}{p} - \subscr{k}{i})(\gamma p - \subscr{k}{p}) B^{-1}G + (\subscr{k}{p} - \subscr{k}{i})(\gamma p - \subscr{k}{p})GB^{-1},
\end{align*}
where in the last inequality we used the fact that the LMI $L^{-1} I_n \preceq B^{-1} \preceq \rho^{-1} I_n$ implies $L^{-1}G^2 I_n \preceq GB^{-1}G \preceq \rho^{-1} G^2$.
Summing up, so far we have 
\begin{align*}
\frac{3 \gamma p}{2}Q_{12}^\top Q_{11}^{-1}Q_{12} & \preceq \gamma^2 \subscr{k}{p}^2 L G^2 + 2 \gamma \subscr{k}{p}(\subscr{k}{p} - \subscr{k}{i}) G + 2\gamma \subscr{k}{p}(\gamma p - \subscr{k}{p}) G^2 + (\subscr{k}{p} - \subscr{k}{i})^2 \rho^{-1}I_n \\
&\quad + (\gamma p - \subscr{k}{p})^2\rho^{-1} G^2 + (\subscr{k}{p} - \subscr{k}{i})(\gamma p - \subscr{k}{p}) B^{-1}G + (\subscr{k}{p} - \subscr{k}{i})(\gamma p - \subscr{k}{p})GB^{-1}.
\end{align*}
Now, since $0 \preceq G \preceq I_n$, to ensure the LMI $Q_{22} - Q_{12}^\top Q_{11}^{-1}Q_{12} \succeq 0$ holds, we must set the negative definite term $-(\subscr{k}{p} - \subscr{k}{i})^2 \rho^{-1}I_n$ to zero, that is set $\subscr{k}{p} = \subscr{k}{i}$. Next, we compute
\begin{align}
\frac{3 \gamma p}{2}\left(Q_{22} - Q_{12}^\top Q_{11}^{-1}Q_{12}\right) \succeq&~ 3 \gamma^2 p \subscr{k}{p}G - \gamma^2 \subscr{k}{p}^2 L G^2 - 2\gamma \subscr{k}{p}(\gamma p - \subscr{k}{p}) G^2 - (\gamma p - \subscr{k}{p})^2\rho^{-1} G^2  \nonumber\\
\succeq & \left(3 \gamma^2 p \subscr{k}{p} - \gamma^2 \subscr{k}{p}^2 L - 2\gamma \subscr{k}{p}(\gamma p - \subscr{k}{p}) - (\gamma p - \subscr{k}{p})^2\rho^{-1} \right)G^2 \succeq 0 \label{ineq:p1} \\
\iff & 3 p \gamma^2 \rho \subscr{k}{p} - \gamma^2 \subscr{k}{p}^2 L \rho - 2\gamma \rho \subscr{k}{p}(\gamma p - \subscr{k}{p}) - (\gamma p - \subscr{k}{p})^2 \geq 0 \nonumber \\
\iff & \gamma^2 \rho \subscr{k}{p}(3p - \subscr{k}{p} L) - (\gamma p - \subscr{k}{p})(2\gamma \rho \subscr{k}{p}+ \gamma p - \subscr{k}{p}) \geq 0 \nonumber \\
\Longleftarrow \ & \gamma^2 \rho \subscr{k}{p}(3p - \subscr{k}{p} L) \geq 0 \quad \text{ and } \quad (\gamma p - \subscr{k}{p})(2\gamma \rho \subscr{k}{p} + \gamma p - \subscr{k}{p}) \leq 0, \label{ineq:p2}
\end{align}
where~\eqref{ineq:p1} follows from the inequality $G^2 \preceq G$. Now, we have 
\begin{enumerate}[label=(\roman*)]
\item $\gamma^2 \rho \subscr{k}{p}(3p - \subscr{k}{p}L) \geq 0 \iff p \geq \frac{\subscr{k}{p}L}{3}$, and
\item $(\gamma p - \subscr{k}{p})(2\gamma \rho \subscr{k}{p} + \gamma p - \subscr{k}{p}) \leq 0 \iff \frac{\subscr{k}{p}(1 - 2\gamma \rho)}{\gamma} \leq p \leq \frac{\subscr{k}{p}}{\gamma}$, where we used the fact that $-1 \leq 1 -2\gamma \rho \leq 1$ being $\gamma \rho \leq \frac{\rho}{L} \leq 1$.
\end{enumerate}
Summing up, inequalities~\eqref{ineq:p2} are satisfied for any $p \in \left[\max\left(\frac{\subscr{k}{p}L}{3}, \frac{\subscr{k}{p}(1 - 2\gamma \rho)}{\gamma}\right), \frac{\subscr{k}{p}}{\gamma}\right]$.
This concludes the proof.
\end{proof}

Next, we show that the PI--PGD globally converges to an equilibrium point.
\bt[Global convergence of~\eqref{eq:pi-pgd_affine}]
\label{thm:GLin-ExpS}
Consider the PI--PGD~\eqref{eq:pi-pgd_affine} with $\subscr{k}{p} = \subscr{k}{i} >0$, satisfying Assumptions~\ref{ass:1_f} and~\ref{ass:2_A}. Let $\gamma \in \left]0, \frac{1}{L}\right]$, $p = \frac{\subscr{k}{p}}{\gamma}$,
and $P$ be as in~\eqref{eq:matrix_P}.
Let $\mathcal{Z}^\star := \{\xstar\} \times \Lambda^\star$, where $\xstar$ is the unique minimizer of~\eqref{eq:eq_constrained_non_smooth} and $\Lambda^\star$ is the associated set of Lagrange multipliers.
Then every solution $z(t) = (x(t), \lambda(t))$ of the PI--PGD~\eqref{eq:pi-pgd_affine} satisfies $x(t) \to \xstar$ as $t \to \infty$, and
\begin{enumerate}[label=\textup{(\roman*)}]
\item $z(t)$ converges to $\bar z = (\xstar, \bar\lambda) \in \mathcal{Z}^\star$; \label{item:global_conv}
\item if $\zstar := (\xstar, \lambda^\star)$ is a locally exponentially stable equilibrium of~\eqref{eq:pi-pgd_affine}, then there exists a norm $\norm{\cdot}_{\loc}$ such that $z(t)$ linear-exponentially converges towards $\zstar$ with respect to the norms $\norm{\cdot}_{P}$ and $\norm{\cdot}_{\loc}$.\label{item:lin_exp}
\end{enumerate}
\et
\begin{proof}
By Lemma~\ref{lem:equivalence}, the equilibria of the PI--PGD~\eqref{eq:pi-pgd_affine} are the stationary points of~\eqref{eq:eq_constrained_non_smooth}. Under Assumption~\ref{ass:1_f}, the cost is strongly convex and thus the primal minimizer $\xstar$ is unique. Assumption~\ref{ass:2_A} ensures that the dual set $\Lambda^\star$ is non-empty. Therefore, the equilibrium set is $\mathcal{Z}^\star = \{\xstar\}\times\Lambda^\star$.

By Lemma~\ref{thm:weak-contractivity}, for the considered $\gamma$ and $p$, the PI-PGD~\eqref{eq:pi-pgd_affine} is weakly infinitesimally contracting with respect to the norm $\norm{\cdot}_{P}$. Then,~\cite[Theorem 4.4]{FB:24-CTDS} guarantees that every equilibrium $z^\star$ is stable with weak Lyapunov function $V(z) = \frac{1}{2}\norm{z - \zstar}_{P}$.
This in turn implies that the trajectory $z(t)$ is bounded for all $t \ge 0$ and since $V(z(t))$ is bounded below by zero and is monotonically non-increasing, the limit $l := \lim_{t \to \infty} \norm{z(t) - \zstar}_{P}$ exists for any $\zstar \in \mathcal{Z}^\star$.

Because the trajectory is bounded, it must converge to a non-empty, compact, positively invariant $\omega$-limit set, which we denote as $\Omega$.
Let $\bar{z}\in\Omega$ be arbitrary and let $\bar{z}(t) = (\bar{x}(t), \bar{\lambda}(t))$ denote the solution of~\eqref{eq:pi-pgd_affine} with $\bar{z}(0)=\bar{z}$. By the positive invariance of $\Omega$, we have $\norm{\bar{z}(t)-\zstar}_P^2 = l$ for all $t\ge 0$. Differentiating this constant yields, for all $t \geq 0$, $0 = \frac{1}{2}\tfrac{d}{dt}\norm{\bar{z}(t)-\zstar}_P^2 = \bigl(\bar{z}(t)-\zstar\bigr)^{\top} P \subscr{F}{PGD}(\bar{z}(t))$.

Next, we show $\frac{d}{dt}\norm{\bar{z} - z^\star}_P^2 \leq - p \norm{\dot{x}(\bar{z})}^2$.
Let $e_x = \bar{x} - \xstar$, $e_\lambda = \bar{\lambda} - \lambda^\star$, and $\Delta f = \nabla f(\bar{x}) - \nabla f(\xstar)$. Evaluating the derivative of $V$ along $\bar{z}(t)$ for $\subscr{k}{p} = \subscr{k}{i} = k$ yields
\beq
\frac{1}{2}\frac{d}{dt}\norm{\bar{z} - z^\star}_P^2 = p e_x^\top \dot{\bar{x}} + e_\lambda^\top \dot{\bar{\lambda}} = p e_x^\top \dot{\bar{x}} + k(A^\top e_\lambda)^\top(\dot{\bar{x}} + e_x),
\label{eq:bound_Vdot}
\eeq
Let $\bar y = \bar x - \gamma(\nabla f(\bar x) + A^\top \bar \lambda)$ and $y^\star = \xstar - \gamma(\nabla f(\xstar) + A^\top \lambda^\star)$. Applying the firm non-expansiveness of the proximal operator (that is, $(\prox{\gamma g}(\bar y) - \prox{\gamma g}(y^\star))^\top(\bar y - y^\star) \ge \norm{\prox{\gamma g}(\bar y) - \prox{\gamma g}(y^\star)}^2$) and substituting the algebraic identities $\prox{\gamma g}(\bar y)  = \bar x + \dot{\bar x}$ and $\prox{\gamma g}(y^\star)  = x^\star$,  yields:
$
(e_x + \dot{\bar x})^\top \bigl(e_x - \gamma \Delta f  - \gamma A^\top e_{\lambda}\bigr) \geq \norm{e_x + \dot{\bar x}}^2.
$
Expanding both sides gives
$
e_x^\top \dot{\bar x} \leq -\norm{\dot{\bar x}}^2 - \gamma \Delta f^\top(e_x + \dot{\bar x}) - \gamma(A^\top e_\lambda)^\top(e_x + \dot{\bar x}).
$
Moreover, we have
\begin{align}
-\Delta f^\top(e_x + \dot{x}) &= -\Delta f^\top e_x - \Delta f^\top \dot{x}\nonumber \\
&\leq -\frac{1}{L}\norm{\Delta f}^2 + \norm{\Delta f}\norm{\dot{x}} \label{ineq:cs}\\
&\leq -\frac{1}{L}\norm{\Delta f}^2 + \left( \frac{1}{2L}\norm{\Delta f}^2 + \frac{L}{2}\norm{\dot{x}}^2 \right) \label{ineq:yineq}\\
&= -\frac{1}{2L}\norm{\Delta f}^2 + \frac{L}{2}\norm{\dot{x}}^2 \leq \frac{L}{2}\norm{\dot{x}}^2,\nonumber
\end{align}
where in~\eqref{ineq:cs} we used $L$-smoothness and convexity of $f(x)$, which guarantees that $\nabla f(x)$ is $L$-co-coercive (i.e., $\Delta f^\top e_x \ge \frac{1}{L}\norm{\Delta f}^2$) and Cauchy-Schwarz, while in~\eqref{ineq:yineq} we used Young's inequality.

Then, using the above inequality and $p = \subscr{k}{p}/\gamma$ we can bound~\eqref{eq:bound_Vdot} as follows
\begin{align*}
\frac{1}{2}\frac{d}{dt}\norm{\bar z - z^\star}_P^2 &\leq -p\norm{\dot{\bar x}}^2 - p\gamma \Delta f^\top(e_x+\dot{\bar x})\\
&\leq -p \left(1 - \frac{\gamma L}{2} \right)\norm{\dot{\bar x}}^2 \leq -\frac{p}{2}\norm{\dot{\bar x}}^2,
\end{align*}
where in the last inequality we used the upper bound $\gamma \leq 1/L$.

Summing up, we have shown that $ 0 = \frac{d}{dt}\norm{\bar{z}(t)-z^\star}_P^2 \le -p \norm{\dot{x}(\bar{z}(t))}^2 \le 0,$
which implies $\dot{x}(\bar{z}(t))=\0_n$, for all $t\ge 0$.
Therefore, the primal trajectory is constant, i.e.,
$
\bar{x}(t)\equiv \hat{x}
$
for some $\hat{x}\in\R^n$. Substituting $\dot{x}(\bar{z}(t))=\0_n$ and $\bar{x}(t)=\hat{x}$ into the second component of~\eqref{eq:pi-pgd_affine} gives
$\dot{\bar{\lambda}}(t) = k(A\hat{x}-b)$.
Hence, $\dot{\bar{\lambda}}(t)$ is constant. Because $\bar{z}(t)\in\Omega$ and $\Omega$ is compact, the trajectory $\bar{\lambda}(t)$ must remain bounded. Since the integral of a nonzero constant vector grows unboundedly with time, $\bar{\lambda}(t)$ remains bounded if and only $k(A\hat{x}-b)=\0_m$, and thus $A\hat{x}=b$. Moreover, since $\dot{x}(\bar{z}(t))=\0_n$, the primal dynamics imply $\hat{x} = \prox{\gamma g} \left(\hat{x} - \gamma\bigl(\nabla f(\hat{x}) + A^\top \bar{\lambda}(t)\bigr) \right)$, for all $t\ge0$.
By the characterization of the proximal operator, this is equivalent to
$ \0_n \in \nabla f(\hat{x}) + \partial g(\hat{x}) + A^\top \bar{\lambda}(t)$, which together with the feasibility condition $A\hat{x}=b$, implies that the pair
$(\hat{x},\bar{\lambda}(t))$ is a stationary point for~\eqref{eq:eq_constrained_non_smooth}. By Lemma~\ref{lem:equivalence}, $\hat{x}$ is therefore a primal minimizer of~\eqref{eq:eq_constrained_non_smooth}. Since the minimizer is unique under Assumption~\ref{ass:1_f}, we conclude that $\hat{x}=x^\star$. Hence, $\bar{z}(t) = (x^\star,\bar{\lambda}) \in \mathcal{Z}^\star$, for all $t\ge0$.

Now, fix any $\bar{z}\in\Omega\subseteq\mathcal{Z}^\star$. Weak contractivity of~\eqref{eq:pi-pgd_affine} implies that $t\mapsto\norm{z(t)-\bar{z}}_P$ is non-increasing. Because $\bar{z}\in\Omega$, there exists a sequence $t_k\to+\infty$ with $z(t_k)\to\bar{z}$, so $\norm{z(t_k)-\bar{z}}_P\to 0$. A non-negative, non-increasing real function that admits a subsequence converging to zero must itself converge to zero. Therefore $\norm{\bar{z}(t)-\bar{z}}_P\to 0$, i.e., $z(t)\to\bar{z}=(\xstar,\bar\lambda)\in\mathcal{Z}^\star$ and, in particular $x(t)\to\xstar$. This proves statement~\ref{item:global_conv}.

Next to prove item~\ref{item:lin_exp}, note that since $\zstar$ is locally exponentially stable, then there exists a norm $\norm{\cdot}_{\loc}$, $\ce >0$ such that the system~\eqref{eq:pi-pgd_affine} is strongly infinitesimally contracting in a neighborhood of the equilibrium point, say it  $\ball{\loc}{z^\star, r}$~\cite{TS:75}. The statement then follows by applying Theorem~\ref{thm:local-exp-contractivity_same_norms}. This concludes the proof.
\end{proof}

\begin{rem}
In the experiments in Section~\ref{ex:lasso}, we numerically validated the local exponential stability assumption of the PI--PGD equilibrium in
Theorem~\ref{thm:GLin-ExpS}. We leave the analysis of finding explicit conditions implying local exponential stability of the equilibrium point to future work.
\end{rem}

\section{Numerical Examples and Applications}
We demonstrate the effectiveness of the PI--PGD in solving constrained composite OPs via two applications: (i) constrained $\ell_1$-regularized least squares problem, also known as LASSO, and (ii) {nonlinear-equality-constrained LASSO}.

\subsection{Equality-Constrained LASSO}
\label{ex:lasso}
Consider the following equality-constrained composite optimization problem:
\beq
\label{eq:constrained_lasso}
\begin{aligned}
\min_{x \in \R^n} \ & \frac{1}{2} x^\top W x + \alpha \norm{x}_1\\
& \text{s.t. } Ax = b,
\end{aligned}
\eeq
where $\alpha >0$, $W \in \R^{n\times n}$ is positive definite, and thus the function $f(x) = \frac{1}{2} x^\top W x$ satisfies Assumption~\ref{ass:1_f}, and $A\in \R^{m \times n}$ satisfies Assumption~\ref{ass:2_A}. Recalling that $\prox{\gamma \alpha \norm{x}_1} = \soft{\gamma \alpha}$, where $\map{\operatorname{soft}_{\gamma \alpha}}{\R^n}{\R^n}$ is the \emph{soft thresholding function} defined by $(\softt{\gamma \alpha}{x})_i= \softt{\gamma \alpha}{x_i}$, and the map $\map{\operatorname{soft}_{\gamma \alpha}}{\R}{\R}$ is defined by
\[
\softt{\gamma \alpha}{x_i} =
\left\{
\begin{aligned}
&x_i - \gamma \alpha\sign{x_i} && \textup{ if } \abs{x_i} > \gamma \alpha,\\
&0 && \textup{ if } \abs{x_i} \leq \gamma \alpha,
\end{aligned}
\right.
\]
with $\map{\operatorname{sign}}{\R}{\{-1,0,1\}}$ being the \emph{sign function} defined by $\sign{x_i} := -1$ if $x_i<0$, $\sign{x_i} := 0$ if $x_i=0$, and $\sign{x_i} :=1$ if $x_i>0$.
Now and throughout the rest of the paper, we slightly abuse notation by using the same symbol for both the scalar and vector forms of the soft-thresholding operator.

The PI--PGD dynamics associated with problem~\eqref{eq:constrained_lasso} are
\beq
\label{eq:pipgd_constrained_lasso}
\begin{cases}
\dot{x} = -x + \soft{\gamma \alpha}\!\bigl((I_n - \gamma W)x - \gamma A^{\top}\lambda\bigr),\\[4pt]
\dot{\lambda} = (\subscr{k}{i} - \subscr{k}{p})Ax + \subscr{k}{p}A\soft{\gamma \alpha}\!\bigl((I_n - \gamma W)x - \gamma A^{\top}\lambda\bigr) - \subscr{k}{i}b.
\end{cases}
\eeq

For the simulations, we set $n=10$, $m=5$, $\alpha=1$, and $W = 10I_n + \tilde{W}\tilde{W}^\top \succ 0$, with $\tilde{W}$, $A$, and $b$ having independent, normally distributed components. We solved~\eqref{eq:constrained_lasso} using \texttt{cvxpy} to obtain the reference solution $z^\star = (x^\star, \lambda^\star)$. Next, we simulate~\eqref{eq:pipgd_constrained_lasso} over the time interval $t \in [0, 20]$ using a forward Euler discretization with stepsize $\Delta t = 0.01$, starting from random initial conditions. In accordance with Theorem~\ref{thm:GLin-ExpS}, we set $\gamma = 1/L$, $\subscr{k}{i} = \subscr{k}{p} = 20$, $p = \subscr{k}{p}/\gamma$, and $P$ as in~\eqref{eq:matrix_P}, where $\rho$ and $L$ are the minimum and maximum eigenvalues of $W$, respectively.
Simulations confirm that $z^\star$ is locally exponentially stable, in accordance with our results, and that the trajectories converge to $z^\star$. We plot the trajectories of the dynamics along with the reference values found using \texttt{cvxpy}, and the constraint $A x - b$ over time in Figure~\ref{fig:trajectories+constraints}.
\begin{figure}[!h]
    \centering
    \includegraphics[width=0.75\linewidth]{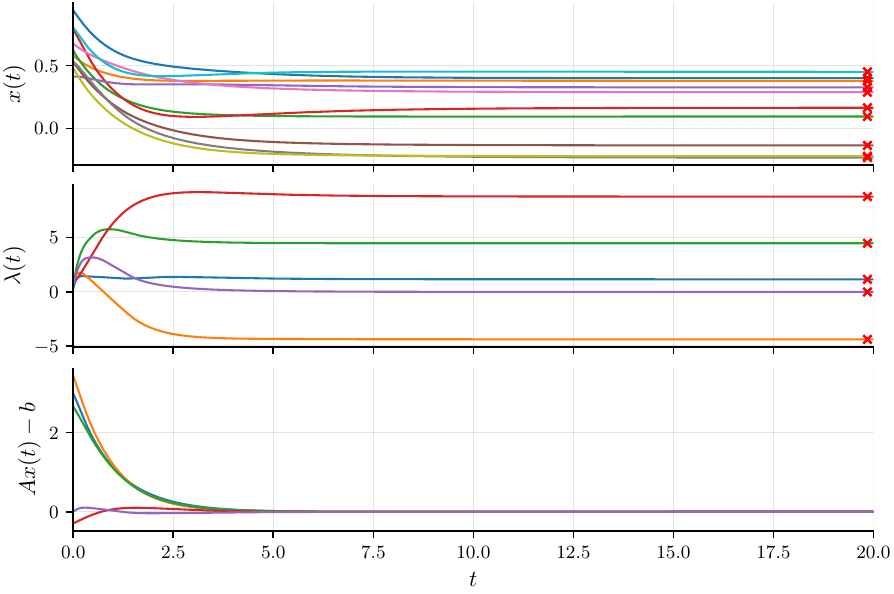}
    \caption{Trajectories of the dynamics~\eqref{eq:pipgd_constrained_lasso} solving the constrained minimization problem~\eqref{eq:constrained_lasso}. The figure shows the trajectories of the primal variables $x(t)$ (top) and two dual variables $\lambda(t)$ (middle), starting from $z^1_0$ and $z^2_0$ as solid and dashed curves, respectively. The \texttt{cvxpy} reference values are shown as dots. The bottom panel displays the constraint residual $A x(t) - b$ over time.}
\label{fig:trajectories+constraints}
\end{figure}
Finally, figure~\ref{fig:lognorm_distance} illustrates the mean and standard deviation of the lognorm of the $\norm{\cdot}_P$ distance across 150 simulated trajectories of~\eqref{eq:pipgd_constrained_lasso} w.r.t. $z^\star$. In agreement with Theorem~\ref{thm:GLin-ExpS}, the figure shows that convergence is linearly-exponentially bounded.

\begin{figure}[!h]
\centering
\includegraphics[width=0.75\linewidth]{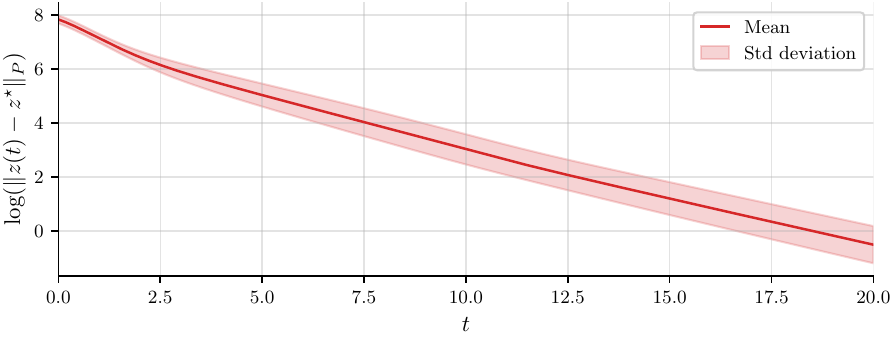}
\caption{Mean and standard deviation of $\log(\norm{z(t) - z^\star}_{P})$ across 150 simulations. Consistent with Theorem~\ref{thm:GLin-ExpS}, convergence is linearly-exponentially bounded.}
\label{fig:lognorm_distance}
\end{figure}
To investigate the influence of the non-smooth regularization term, we repeat the experiment for different values of $\alpha \in \{10^{-2}, 10^{-1}, 1, 10, 30, 50\}$. 
For each $\alpha$, we recompute the reference solution $z^\star$ using \texttt{cvxpy} and simulate~\eqref{eq:pipgd_constrained_lasso} from identical initial conditions. Figure~\ref{fig:nonsmoothness} illustrates the evolution of $\log(\norm{z(t)-z^\star}_P)$. The results show that convergence remains consistent with the theoretical predictions across the entire range of $\alpha$. However, for large values of $\alpha$, the dynamics exhibits deviations from a purely linear-exponential rate. This behavior can be attributed to the increased influence of the $\ell_1$-regularization term causing the trajectories to spend more time near non-differentiable regions of the soft-thresholding operator.
\begin{figure}[!h]
\centering
\includegraphics[width=.75\linewidth]{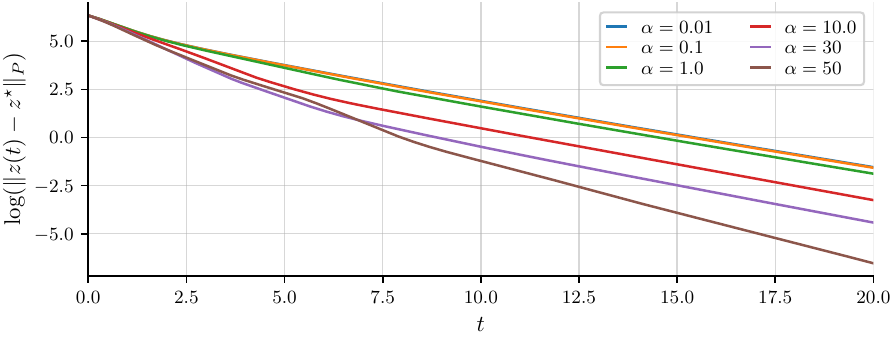}
\caption{Evolution of $\log(\norm{z(t)-z^\star}_P)$ for $\alpha \in \{10^{-2}, 10^{-1}, 1, 10, 30, 50\}$. Larger values of $\alpha$ lead to increased influence of non-differentiable regions.}
\label{fig:nonsmoothness}
\end{figure}

\paragraph{Comparison with the Proximal Augmented Lagrangian Gradient Flow in~\cite{NKD-SZK-MRJ:19}.}
\label{sec:comparison_pal}
We compare the proposed PI--PGD~\eqref{eq:pipgd_constrained_lasso} with the continuous-time primal--dual flow derived from the \emph{proximal augmented Lagrangian} (PAL) in~\cite{NKD-SZK-MRJ:19}.
The PAL dynamics addresses composite optimization problems of the form
\beq
\label{eq:pal_composite}
\min_{x} f(x) + g(\mathcal{T}(x)),
\eeq
where $f$ is continuously differentiable, $g$ is convex, closed, and proper, and $\mathcal{T}$ is a bounded linear operator. The PAL is
\beq
\label{eq:pal-gf-general}
\begin{cases}
\dot x = -\nabla f(x) - \mathcal{T}^\top \nabla M_{\mu g}\bigl(\mathcal{T}(x)+\mu y\bigr),\\[2pt]
\dot y = \mu\Bigl(\nabla M_{\mu g}\bigl(\mathcal{T}(x)+\mu y\bigr) - y\Bigr).
\end{cases}
\eeq
where $M_{\mu g}$ denotes the Moreau envelope of $g$ with parameter $\mu$~\cite{HHB-PLC:17}. Under Assumptions~\ref{ass:1_f}--\ref{ass:2_A}, and $g$ convex, closed, and proper,~\eqref{eq:pal-gf-general} is globally exponentially stable for every $\mu \ge L - \rho$~\cite[Thm.~3]{NKD-SZK-MRJ:19}.

Next, we show that the equality-constrained optimization problem~\eqref{eq:eq_constrained_non_smooth}, can be equivalently written as in the form~\eqref{eq:pal_composite}. To this end, we follow the two-block construction of~\cite[Remark~1]{NKD-SZK-MRJ:19}. Define the stacked operator $\map{\mathcal{T}}{\R^n}{\R^{n+m}}$,
$$
\mathcal{T}(x) := \begin{pmatrix} x \\ Ax \end{pmatrix},
$$
and the separable non-smooth term $\map{\tilde g}{\R^n\times\R^m}{\R\cup\{+\infty\}}$, $\tilde g(z_1,z_2) := g(z_1) + \iota_{\{b\}}(z_2)$, where $\iota_{\{b\}}$ is the indicator function of the singleton $\{b\}\subset\R^m$. Problem~\eqref{eq:eq_constrained_non_smooth} is then equivalent to $\min_x f(x)+\tilde g(\mathcal{T}(x))$. Since $\prox{\mu \iota_{\{b\}}}(v_2)=b$ for every $\mu>0$, $\tilde g$ is separable and so is its Moreau envelope and we have
$$
M_{\mu \tilde g}(v_1,v_2) = M_{\mu g}(v_1) + \frac{1}{2\mu}\|v_2-b\|^2,
\qquad
\nabla M_{\mu \tilde g}(v_1,v_2) =
\begin{pmatrix}
\nabla M_{\mu g}(v_1) \\
\frac{1}{\mu}(v_2-b)
\end{pmatrix}.
$$
Writing the dual variable as $y=(y_1,y_2)\in\R^n\times\R^m$, the PAL solving the equality constrained problem~\eqref{eq:eq_constrained_non_smooth} is 
\beq
\label{eq:pal-gf}
\begin{cases}
\dot x = -\nabla f(x) - \nabla M_{\mu g}(x+\mu y_1) - A^\top \left(\dfrac{1}{\mu}(Ax-b) + y_2\right),\\
\dot y_1 = \mu\bigl(\nabla M_{\mu g}(x+\mu y_1) - y_1\bigr),\\
\dot y_2 = Ax-b.
\end{cases}
\eeq

In particular, for the equality-constrained LASSO~\eqref{eq:constrained_lasso}, $f(x)= \tfrac12 x^\top W x$ and $g(x)=\alpha\|x\|_1$, so that $\nabla M_{\mu g}(v) = \frac{1}{\mu} \sat{\mu\alpha}{v}$, and the PAL dynamics~\eqref{eq:pal-gf} reads
\beq
\label{eq:pal-gf_lasso}
\begin{cases}
\dot x = - Wx - \frac{1}{\mu} \sat{\mu\alpha}{x+\mu y_1} - A^\top \left(\dfrac{1}{\mu}(Ax-b) + y_2\right),\\
\dot y_1 = \sat{\mu\alpha}{x+\mu y_1} - \mu y_1,\\
\dot y_2 = Ax-b,
\end{cases}
\eeq
with $\mu \geq L-\rho$, $L=\lambda_{\max}(W)$, $\rho=\lambda_{\min}(W)$, guaranteeing global exponential stability of $(x^\star,y_1^\star,y_2^\star)$~\cite[Thm.~3]{NKD-SZK-MRJ:19}.
Both~\eqref{eq:pi-pgd_affine} and~\eqref{eq:pal-gf_lasso} are continuous-time primal-dual flows enjoying global convergence guarantees under Assumptions~\ref{ass:1_f}--\ref{ass:2_A}, but the two approaches differ in two key respects. First, PI-PGD evolves on the state space $\R^{n+m}$, whereas PAL-GF introduces an auxiliary variable $y_1\in\R^n$, yielding a larger state dimension of $2n+m$. Second, PI-PGD is derived from a control-theoretic perspective, while PAL-GF originates from an optimization-theoretic construction, smoothing the non-smooth term $g$ through its Moreau envelope within a proximal augmented Lagrangian.

We simulate~\eqref{eq:pal-gf_lasso} with $\mu \geq L-\rho$ alongside \eqref{eq:pipgd_constrained_lasso} on the same problem instance as above. Each method is discretized via forward Euler over the common horizon $t\in[0,20]$ with stepsize $\Delta t = 10^{-3}$. Figure~\ref{fig:lognorm_distance_comparison} shows the mean and standard deviation, over $150$ random initializations, of the log-distance $\log\|x(t)-x^\star\|$ together with the constraint residual $\log\|Ax(t)-b\|$. 
Both methods converge to the reference solution $x^\star$, as evidenced by the approximately linear decay observed on the logarithmic scale. Following an initial transient, and for the tested parameter settings, PI-PGD exhibits a steeper decay rate in both the optimality and feasibility metrics. The narrow confidence bands indicate that this behavior is robust across the considered random initializations.
\begin{figure}[!h]
\centering
\includegraphics[width=0.8\linewidth]{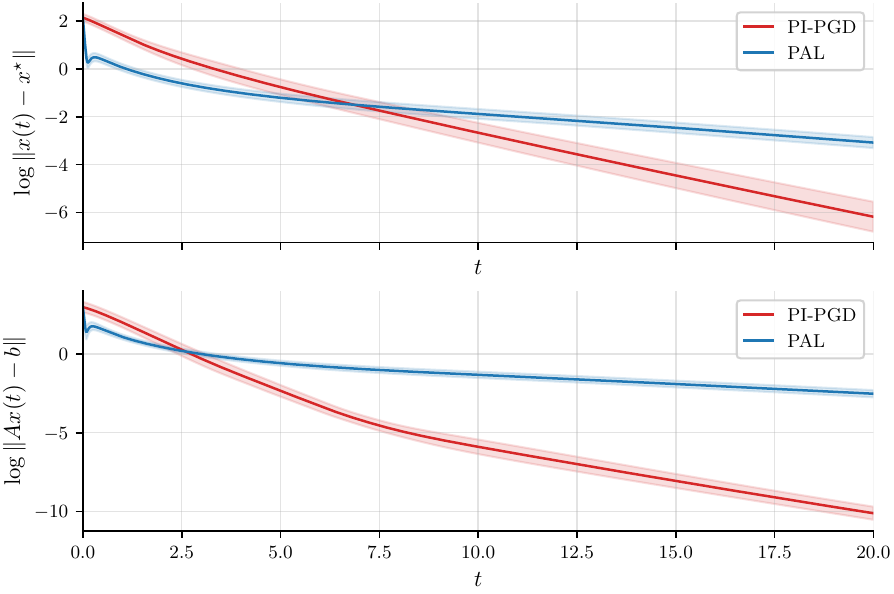}
\caption{Mean and standard deviation over $150$ random initializations of the distances $\log(\norm{x(t) - x^\star})$ (top) and the constraint residual $\log\|Ax(t)-b\|$ (bottom) across 150 simulations for PI-PGD~\eqref{eq:pipgd_constrained_lasso} and PAL-GF~\eqref{eq:pal-gf_lasso}. While both methods converge to $x^\star$, PI-PGD exhibits a consistently steeper decay rate across both quantities.}
\label{fig:lognorm_distance_comparison}
\end{figure}
To assess robustness beyond a single problem instance, we repeat the comparison across $15$ independently generated instances obtained by varying the random seed used to draw $W$, $A$, and $b$, each integrated from $3$ distinct random initial conditions. Each method is discretized via forward Euler over the common horizon $t\in[0,25]$ with stepsize $\Delta t = 10^{-3}$. For every instance the parameters of PI-PGD~\eqref{eq:pipgd_constrained_lasso} and PAL-GF~\eqref{eq:pal-gf_lasso} with $\mu = L-\rho$ are tuned independently.
Figure~\ref{fig:exp2_convergence} reports the mean and standard deviation of the log primal error  $\log\|x(t)-x^\star\|$ (top panel) and the constraint residual  $\log\|Ax(t)-b\|$ (bottom panel). Both methods converge reliably across the full distribution of problem instances, as confirmed by the monotone decay of  both quantities. After an initial transient, PI-PGD exhibits a consistently steeper decay rate than PAL-GF. The narrow shaded bands confirm that this behavior is stable across both problem instances and initial conditions.
\begin{figure}[!h]
\centering
\includegraphics[width=0.8\linewidth]{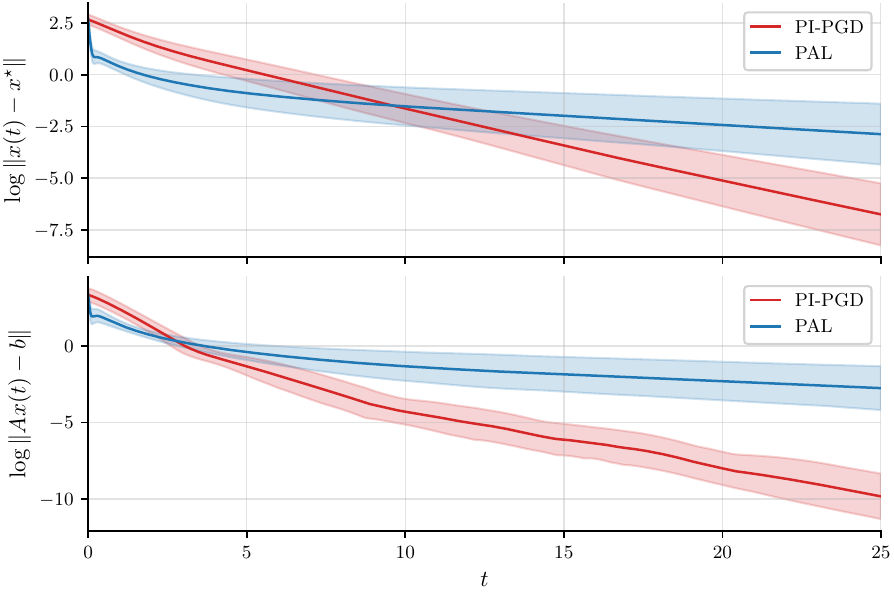}
\caption{Mean and standard deviation over $15$ independently generated instances each integrated from $3$ distinct random initial conditions, of $\log\|x(t)-x^\star\|$ (top) and the constraint residual $\log\|Ax(t)-b\|$ (bottom) for PI-PGD~\eqref{eq:pipgd_constrained_lasso}  and PAL-GF~\eqref{eq:pal-gf_lasso} with $\mu = L-\rho$. PI-PGD achieves a faster asymptotic decay rate across all instances.}
\label{fig:exp2_convergence}
\end{figure}

\subsection{Nonlinear-Equality-Constrained LASSO}
\label{ex:lasso_nonlin}
We next illustrate the applicability of the PI--PGD algorithm to nonconvex composite optimization problems with nonlinear equality constraints. Specifically, we consider the following problem
\beq
\label{eq:opt_problem_non_lin2}
\begin{aligned}
\min_{x \in \R^3} \quad & (x_1 - 1)^2 + (x_2 - 2)^2 + (x_3 + 1)^2 + \alpha \norm{x}_1, \\
\text{s.t.} \quad & h_1(x) := x_1^2 + x_2 - 1 = 0,\\
&h_2(x) := \sin(x_2) + x_3 - 0.5 = 0.
\end{aligned}
\eeq
The Jacobian of the constraints is
\[
D h(x) =
\begin{bmatrix}
2 x_1 & 1 & 0 \\
0 & \cos(x_2) & 1
\end{bmatrix}.
\]
Since the rows of $D h(x)$ are linearly independent for all $x \in \R^3$ (the third column of $D h(x)$ is $[0,1]^\top$), the constraint function $h(x)$ satisfies the LICQ globally.
The corresponding PI--PGD dynamics for problem~\eqref{eq:opt_problem_non_lin2} are
\beq
\label{eq:pi-pgd_non_lin2}
\begin{cases}
\dot{x}_1 = -x_1 + \soft{\gamma \alpha}(x_1 - \gamma (2(x_1 - 1) + 2 \lambda_1 x_1)),\\
\dot{x}_2 = -x_2 + \soft{\gamma \alpha}(x_2 - \gamma (2(x_2 - 2) + \lambda_1 + \lambda_2 \cos(x_2))), \\
\dot{x}_3 = -x_3 + \soft{\gamma \alpha}(x_3 - \gamma (2(x_3 + 1) + \lambda_2)), \\
\dot{\lambda}_1 =  \subscr{k}{p} \big( 2 x_1 \dot{x}_1 + \dot{x}_2 \big) +  \subscr{k}{i} (x_1^2 + x_2 - 1),\\
\dot{\lambda}_2 =  \subscr{k}{p} \big( \cos(x_2) \dot{x}_2 + \dot{x}_3 \big) +  \subscr{k}{i} (\sin(x_2) + x_3 - 0.5).
\end{cases}
\eeq
We solved~\eqref{eq:opt_problem_non_lin2} using the \texttt{SLSQP} solver from the \texttt{SciPy} optimization library to obtain the primal solution $x^\star$, and computed the corresponding multipliers $\lambda^\star$ from the KKT conditions. Next, we simulated the PI--PGD dynamics~\eqref{eq:pi-pgd_non_lin2} over the time interval $t \in [0, 5]$ using \texttt{solve\_ivp}, starting from random initial conditions, with parameters $\alpha =0.5$, $\gamma = 0.5$, $\subscr{k}{i} = 10$, and $\subscr{k}{p} = 15$. Simulations show that the trajectories converge to $z^\star = (x^\star, \lambda^\star)$. We plot the trajectories of the primal and dual variables together with the reference values obtained via \texttt{SLSQP}, and the evolution of the constraint residuals $h(x) = \0_2$ over time in Figure~\ref{fig:trajectories+constraints_nonlin2}. The trajectories of $x(t)$ converge to a neighborhood of the reference optimizer $x^\star$, while the constraints are satisfied after a short settling time. The estimated steady-state multipliers $\lambda(t)$ also approach the reference values $\lambda^\star$ computed from the KKT conditions. Although the theoretical guarantees established in Section~\ref{sec:convergence} do not directly apply to this nonlinear and nonconvex example, these results provide numerical evidence that the PI--PGD dynamics can remain effective beyond the affine/convex regime.
To further illustrate the convergence behavior, we also tracked the cost function $(x_1 - 1)^2 + (x_2 - 2)^2 + (x_3 + 1)^2 + \alpha \norm{x}_1$ over time, shown in Figure~\ref{fig:cost_nonlin2}. The cost decreases monotonically and approaches the reference value computed via \texttt{SLSQP}.
Finally, we also evaluated the effect of different PI gains by simulating the dynamics for several $(\subscr{k}{p}, \subscr{k}{i})$ pairs. Figure~\ref{fig:lognorm_k_nonlin2} shows the distance to the reference solution $\norm{z(t) - z^\star}$ in logarithmic scale for six different gain combinations ranging from $(4,4)$ to $(40,40)$. These plots illustrate how the choice of gains influences convergence speed and transient behavior.
\begin{figure}
\centering
\includegraphics[width=.8\linewidth]{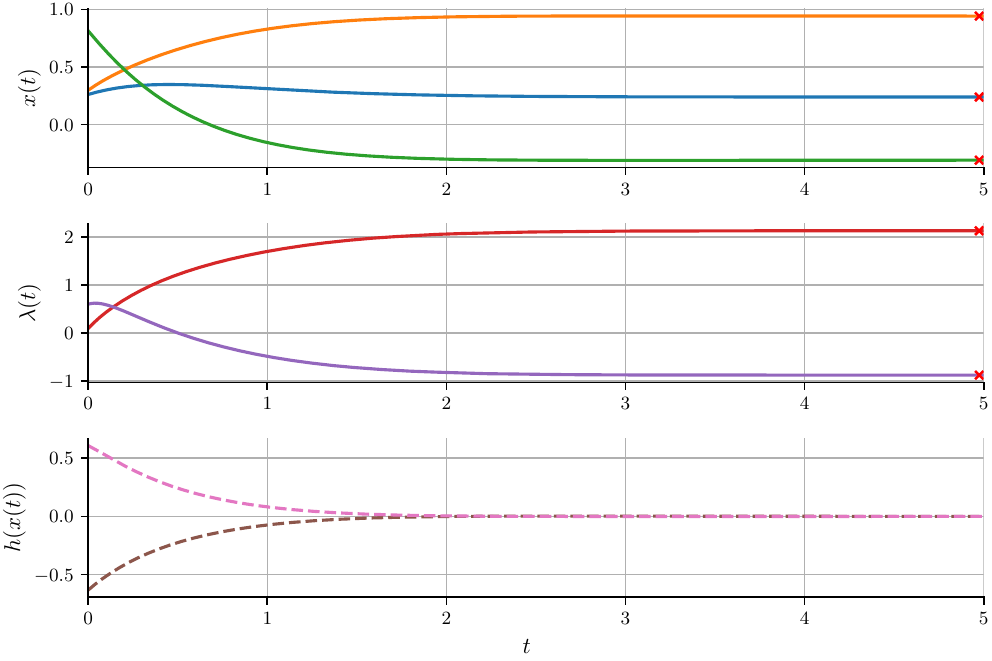}
\caption{Trajectories of~\eqref{eq:pi-pgd_non_lin2} solving LASSO with nonlinear equality constraints~\eqref{eq:opt_problem_non_lin2}. The panel shows the trajectories of the primal variables $x(t)$ (top) and two dual variables $\lambda(t)$ (middle), starting from random initial conditions. The \texttt{SLSQP} reference values are shown as cross. The bottom panel displays the constraint $h(x)$ over time. The trajectories effectively converge to $z^\star$, and constraints are satisfied after a short settling time.}
\label{fig:trajectories+constraints_nonlin2}
\end{figure}
\begin{figure}
\centering
\includegraphics[width=.8\linewidth]{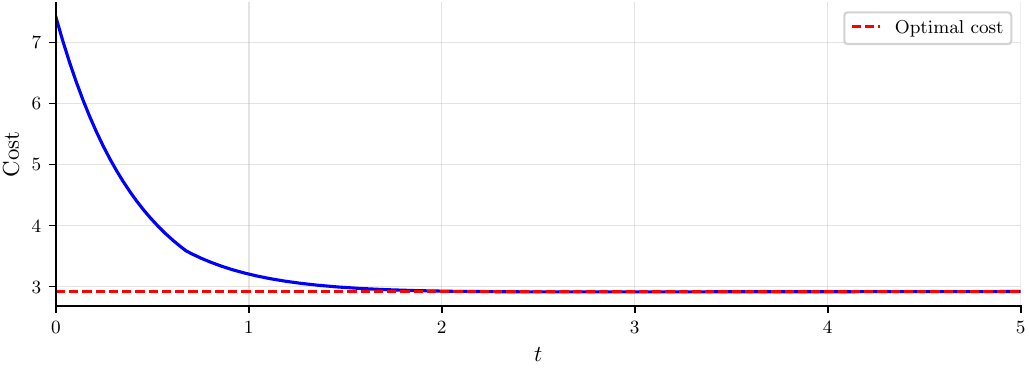}
\caption{Evolution of the cost function over time. The blue curve shows the PI--PGD cost, and the red dashed line indicates the reference cost obtained via \texttt{SLSQP}.}
\label{fig:cost_nonlin2}
\end{figure}
\begin{figure}
\centering
\includegraphics[width=.6\linewidth]{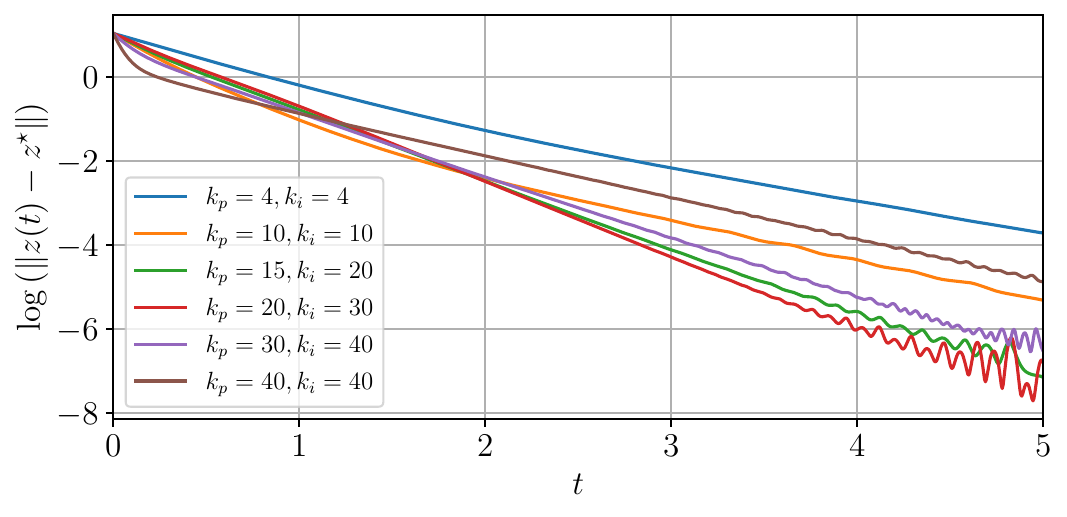}
\includegraphics[width=.6\linewidth]{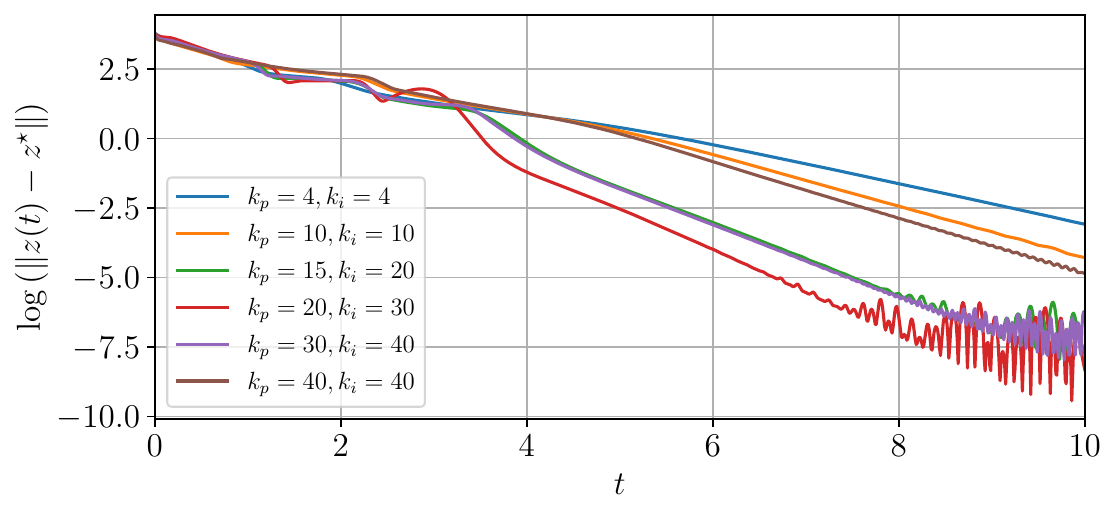}
\includegraphics[width=.6\linewidth]{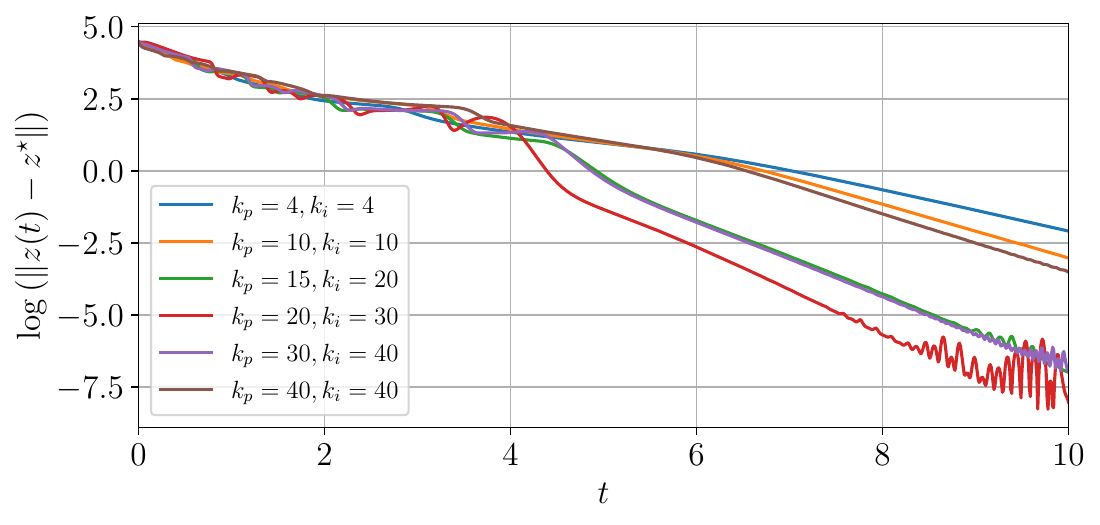}
\caption{Distance to the reference solution $\|z(t) - z^\star \|$ (log scale) for different PI gain pairs $( \subscr{k}{p},  \subscr{k}{i})$. The plot illustrates the effect of gain choice on convergence speed and transient behavior.}
\label{fig:lognorm_k_nonlin2}
\end{figure}

\section{Discussions and Conclusions}
We proposed the PI--PGD~\eqref{eq:pi-pgd} for solving equality-constrained composite OPs. We established the equivalence between the stationary points of the OP and the equilibria of the PI--PGD. For strongly convex and $L$-smooth cost functions and full row-rank affine constraints, we proved global linear-exponential convergence. 
Then, we demonstrated the effectiveness of our approach on equality-constrained LASSO and {further explored its applicability to nonlinear equality constraints}. Remarkably, the PI--PGD converged to the optimal solution even in challenging nonconvex settings. Moreover, for both applications, preliminary numerical experiments suggested that the PI--PGD could still converge to the optimum even when the conditions of Theorem~\ref{thm:GLin-ExpS} were not met.

{Despite these encouraging results, several theoretical aspects of our method remain open.
For composite optimization problems with affine constraints satisfying Assumption~\ref{ass:1_f}--\ref{ass:2_A}, the augmented Lagrangian flows~\cite{NKD-SZK-MRJ:19} is known to be globally strongly contracting~\cite{AD-VC-AG-GR-FB:23f_corrected}. In contrast, the convergence analysis of PI--PGD~\eqref{eq:pi-pgd} is more challenging. At this stage, we have presented the strongest result we could derive, but we believe that tighter guarantees can be obtained through appropriate tuning of the control gains or the design of alternative controllers. 
On the other hand, one of the advantages of the PI--PGD~\eqref{eq:pi-pgd} is that it can be effective in the more challenging setting of nonlinear constraints, as our empirical results suggest. Motivated by the numerical findings,} future work will include:
(i) providing conditions implying local exponential stability of the equilibrium point;
(ii) obtaining sharper convergence bounds and benchmarking our approach against other methods (e.g., the ones from~\cite{AB:17}); (iii) relaxing~\ref{ass:1_f}--\ref{ass:2_A} to cover not-full-rank constraints~\cite{IKO-MRJ:23} and more general constraint structures.
Additionally, we plan to analyze the effects of discretization on the PI-PGD. Indeed, while it is known that for strongly contracting continuous-time dynamics the forward Euler discretization converges under a proper step-size, to the best of our knowledge, a similar result for globally-weakly and locally-strongly contracting dynamics is still missing. Addressing these challenges defines a promising direction for future work.

\appendix
\section{Instrumental Result}
For completeness, we include a self-contained statement of the lemma proposed in~\cite[Lemma~6]{GQ-NL:19}, used in the proof of Lemma~\ref{thm:weak-contractivity}.
\begin{lem}
\label{lemma:BoundQ11}
Let $X = X^\top \in \R^{n\times n}$ satisfy $\xmin \preceq X \preceq \xmax I_n$ for some $\xmax \geq \xmin > 0$, $\gamma > 0$, and $\ds \gamma \leq \frac{1}{\xmax}$. Then for all $g \in [0,1]^n$, the following inequality holds
\beq
\gamma(\diag{g} X+ X\diag{g}) + 2(I_n - \diag{g}) \succ \frac{3}{2}\gamma X.
\eeq
\end{lem}
\begin{proof}
    See~\cite[Lemma~6]{GQ-NL:19}.
\end{proof}

\section{An Application to Entropic Regularized Optimal Transport}
We refer to the standard set-up from, e.g.,~\cite{GP-MC:19}. Given a space $\mcX$, a \emph{discrete probability measure} with weights $a \in \Sigma_n:= \setdef{q \in \R^n_{\geq 0}}{\sum_{i = 1}^n q_i = 1}$ and locations $x_1, \dots, x_n \in \mcX$ is $\alpha(a, x) = \sum_{i= 1}^n a_i \delta_{x_i}$.
Given $\alpha(a, x)$ and $\beta(b, y)$, the set of \emph{admissible coupling} between them is
$
U(a, b):= \setdef{P \in \R^{n\times m}_{+}}{P\1_m = a, P^\top\1_n = b}.
$
The \emph{entropic-regularized optimal transport} problem can be stated as
\beq
\label{eq:entropy_regularized_OP}
\begin{aligned}
\min_{P \in U(a,b)} &\sum_{i,j} P_{ij} C_{ij} + \epsilon \sum_{i,j} P_{ij} \log{(P_{ij})},
\end{aligned}
\eeq
where $C_{ij} := \operatorname{c}(x_i, y_j)$, $\map{\operatorname{c}}{\R^{n} \times \R^m}{\R_{+}}$ is the cost function, $\epsilon > 0$ is a regularization parameter and with the standard convention $0\log{(0)} = 0$. Problem~\eqref{eq:entropy_regularized_OP} has a unique solution as its objective function is $\epsilon$-strongly convex and the set $U(a, b)$ is convex.
However, the objective function is not {$L$-}smooth, since the Hessian is undefined if $P_{ij}=0$, violating Assumption~\ref{ass:1_f}. Despite this, we present a numerical example showing that the PI--PGD successfully solves problem~\eqref{eq:entropy_regularized_OP} and converges even for small $\epsilon$, in cases where the Sinkhorn algorithm fails. To write the PI--PGD we first vectorize~\eqref{eq:entropy_regularized_OP}. To this end, let $p = \vecto(P) \in \R^{nm}$, $c = \vecto(C) \in \R^{nm}$, $d = (a, b)\in \R^{n+m}$, and
$$
A =
\begin{bmatrix}
\1_m^\top \otimes I_n \\
I_m \otimes \1_n^\top 
\end{bmatrix} \in \R^{(n+m)\times nm}.
$$

Then, problem~\eqref{eq:entropy_regularized_OP} becomes
\begin{align*}
&\min_{p \geq 0} p^\top c + \epsilon \sum_{i}^{nm} p_i \log p_i \\
&\quad \text{s.t. } A p = d,
\end{align*}
or equivalently in the form as problem~\eqref{eq:eq_constrained_non_smooth}:
\beq
\label{eq:ot_reg_composite}
\begin{aligned}
&\min_{p \in \R^{nm}} p^\top c + \epsilon \sum_{i}^{nm} p_i \log p_i + \iota_{\R_{\geq 0}^{nm}}(p) \\
&\quad \text{s.t. } A p = d.
\end{aligned}
\eeq
Since $A$ is not full row rank and $a \in \Sigma_n, b\in \Sigma_m$, any one of the $n+m$ equality constraints $A p = d$ is redundant. Thus, without loss of generality, we remove the last constraint and let $\tilde A$ and $\tilde d$ be $A$ without the last row and $d$ without the last entry, respectively. This reduction ensures that $\tilde A$ has full row rank.

{We let $\map{\subscr{f}{OT}}{\R_{\geq 0}^{nm}}{\R}$ be defined by $\subscr{f}{OT}(p) = p^\top c + \epsilon \sum_{i}^{nm} p_i \log p_i$ and make the following:}
\begin{rem}
The indicator function on $\R_{\geq 0}^{nm}$ in~\eqref{eq:ot_reg_composite} enforces {the non-negative constraint and allow us to write the OP~\eqref{eq:entropy_regularized_OP} as in the form of the OP~\eqref{eq:eq_constrained_non_smooth}}.
Therefore, it is reasonable to analyze the function $\subscr{f}{OT}$ only in $\R_{\geq 0}^{nm}$.
\end{rem}

Then, the PI--PGD solving problem~\eqref{eq:ot_reg_composite}
\beq
\label{eq:pi-pgd_ot}
\begin{cases}
\dot p = - p + \relu\left(p - \gamma(\epsilon \log p + \epsilon + c + {\tilde A}^{\top}\lambda)\right) \\
\dot \lambda = (\subscr{k}{i} - \subscr{k}{p}){\tilde A} p + \subscr{k}{p}{\tilde A}\relu\left(p - \gamma(\epsilon \log p + \epsilon + c + {\tilde A}^{\top}\lambda)\right) - \subscr{k}{i}{\tilde d},
\end{cases}
\eeq
where $\lambda \in \R^{n+m-1}$. Note the positive orthant is a forward invariant set for the primal dynamics in~\eqref{eq:pi-pgd_ot}. This ensures that starting with non-negative initial conditions $p_0$, then the primal variable $p(t)$ will remain non-negative.
We validate the PI--PGD~\eqref{eq:pi-pgd_ot} on an image morphing application from~\cite{NB-MVDP-SP-WH:11}. Here, optimal transport is used to transform a picture made of particles into another one by moving the particles. The core idea is to treat each image as a discrete probability distribution, where pixel intensities represent the mass at each location. The transport plan defines how to move this mass from one image to another in an optimal way, considering the cost of moving mass.
In our experiments, the starting {picture represents} number 4 and the final desired {one represents} number 1. Source and target images are from the MNIST dataset. These pictures are stippled using the same number of particles $n=m=100$ with the same mass, therefore $a$ and $b$ are uniform. We define the cost matrix as the Euclidean distance and set $\epsilon=0.001$, $\gamma=0.01$ and $k_p=k_i=100$. The problem is solved using both the PI--PGD method and Sinkhorn algorithm {using the Python Optimal Transport package~\cite{RF-at-all:21}}.
Remarkably, in this set-up with small $\epsilon$, the PI--PGD converges to the optimal transport plan, while {the used standard implementation of Sinkhorn does not converge}. Figure~\ref{fig:OT_comparison_numbers} illustrates the initial, middle, and final frames of the image transformation. The gifs of the complete transformations are available at~\url{https://shorturl.at/mPpEZ}.
\begin{rem}
The Sinkhorn algorithm solves the entropic regularized optimal transport problem using iterative matrix scaling~\cite{MC:13}. It is known that for small regularization parameters $\epsilon$, the algorithm may fail to converge due to numerical instability, as noted in~\cite{GP-MC:19}. Specifically, the main reason is because as $\epsilon$ decreases, the Gibbs kernel $K = \exp(-C / \epsilon)$ becomes numerically ill conditioned leading to potential divergence or stagnation in the iterative process.
\end{rem}
\begin{figure}[!h]
    \centering
    \includegraphics[width=.4\linewidth]{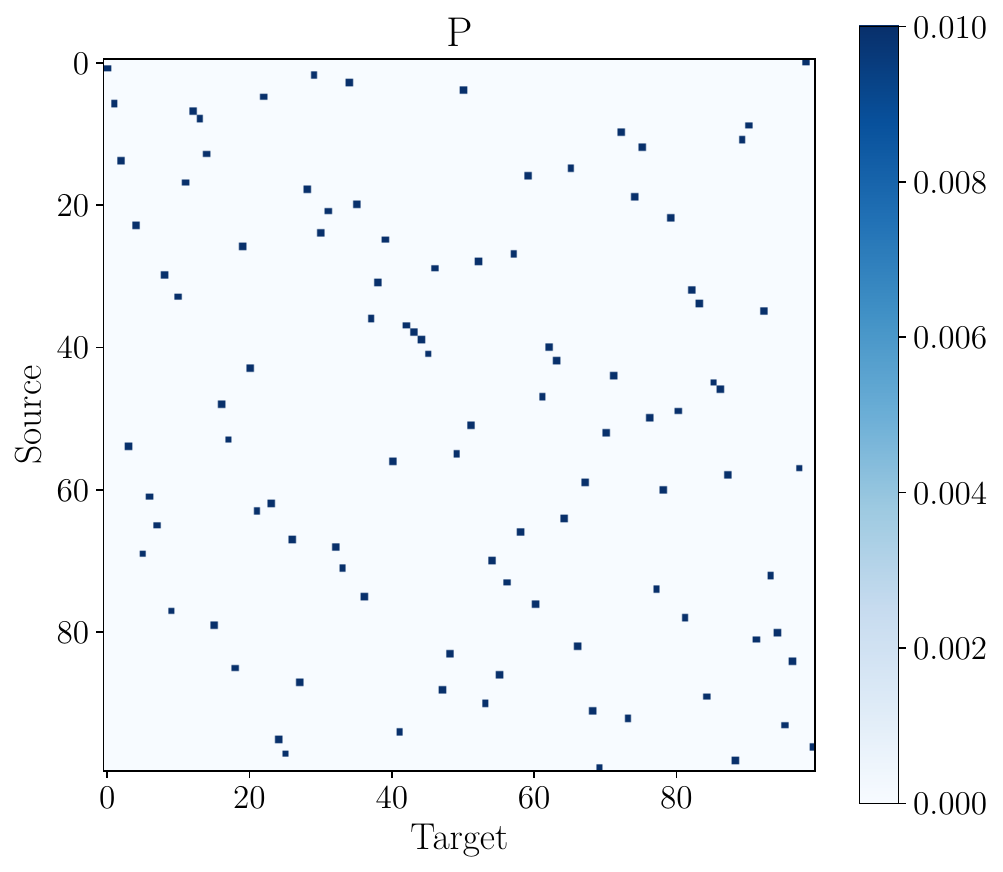}
    \includegraphics[width=.53\linewidth]{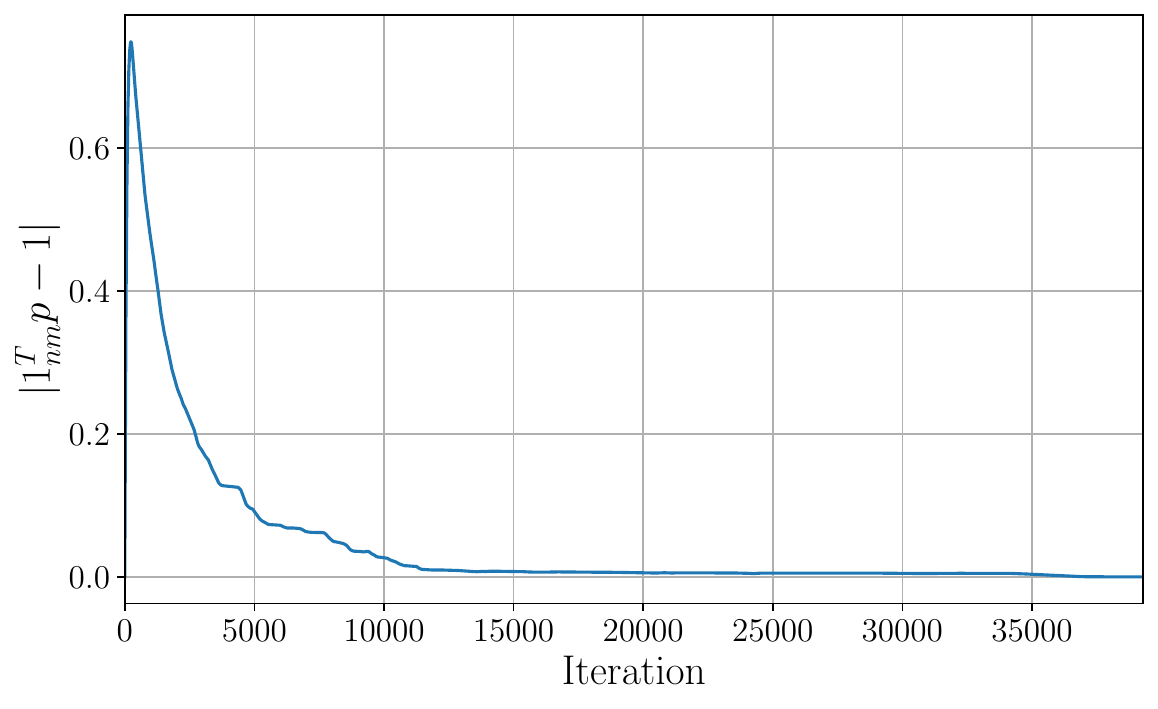}
    \caption{Optimal Transport plan obtained using the PI-PGD~\eqref{eq:pi-pgd_ot} (left). As expected the optimal vector $p \in \R^{nm}$ is a probability vector, that is $p \geq 0$ and $\ds\sum_{j=1}^{nm} p_j = 1$ (right).}
    \label{fig:OT_P_pipgd}
\end{figure}
Figure~\ref{fig:OT_norm_error} shows the norm of the constraint $Ap - b$ over the iterations.
\begin{figure}[!h]
    \centering
    \includegraphics[width=.55\linewidth]{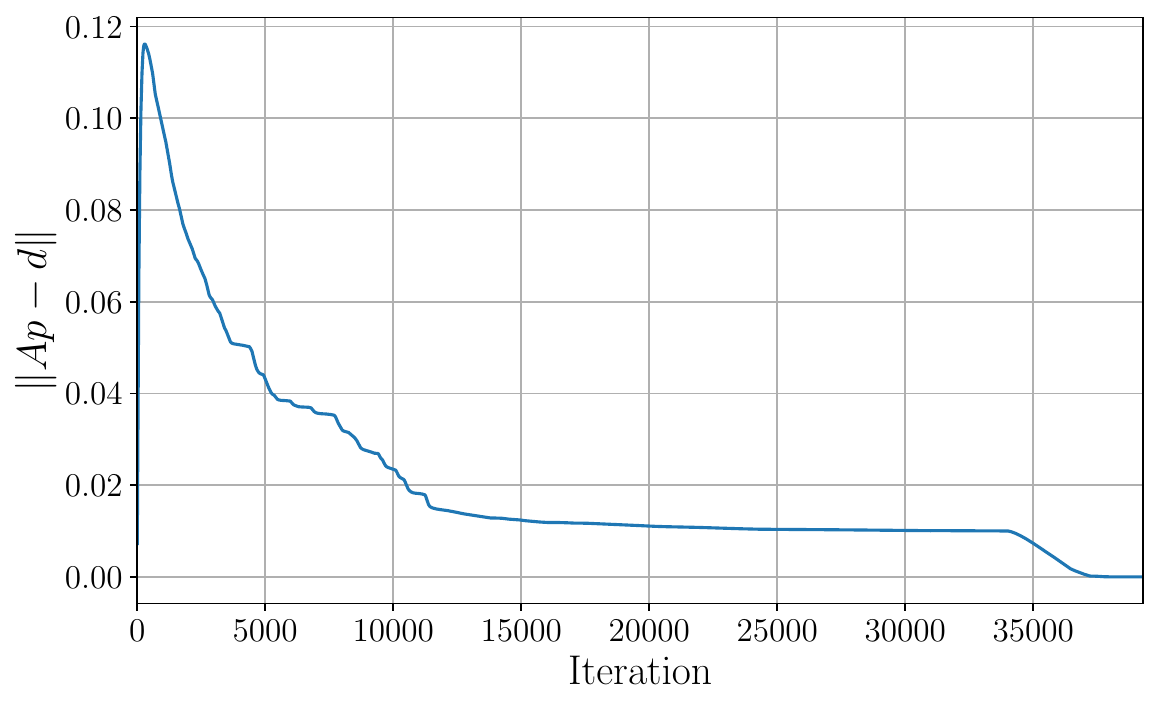}
    \caption{Norm of difference $Ap - d$ over the iterations. In agreement with our results, at convergence, the norm is zero, that is the optimal solution is feasible.}
    \label{fig:OT_norm_error}
\end{figure}
Finally, Figure~\ref{fig:OT_comparison_numbers} illustrates the initial, middle, and final frames of the image transformation. The gifs of the complete transformations are available at~\url{https://shorturl.at/mPpEZ}.
\begin{figure}[!h]
    \centering
    \includegraphics[width=.6\linewidth]{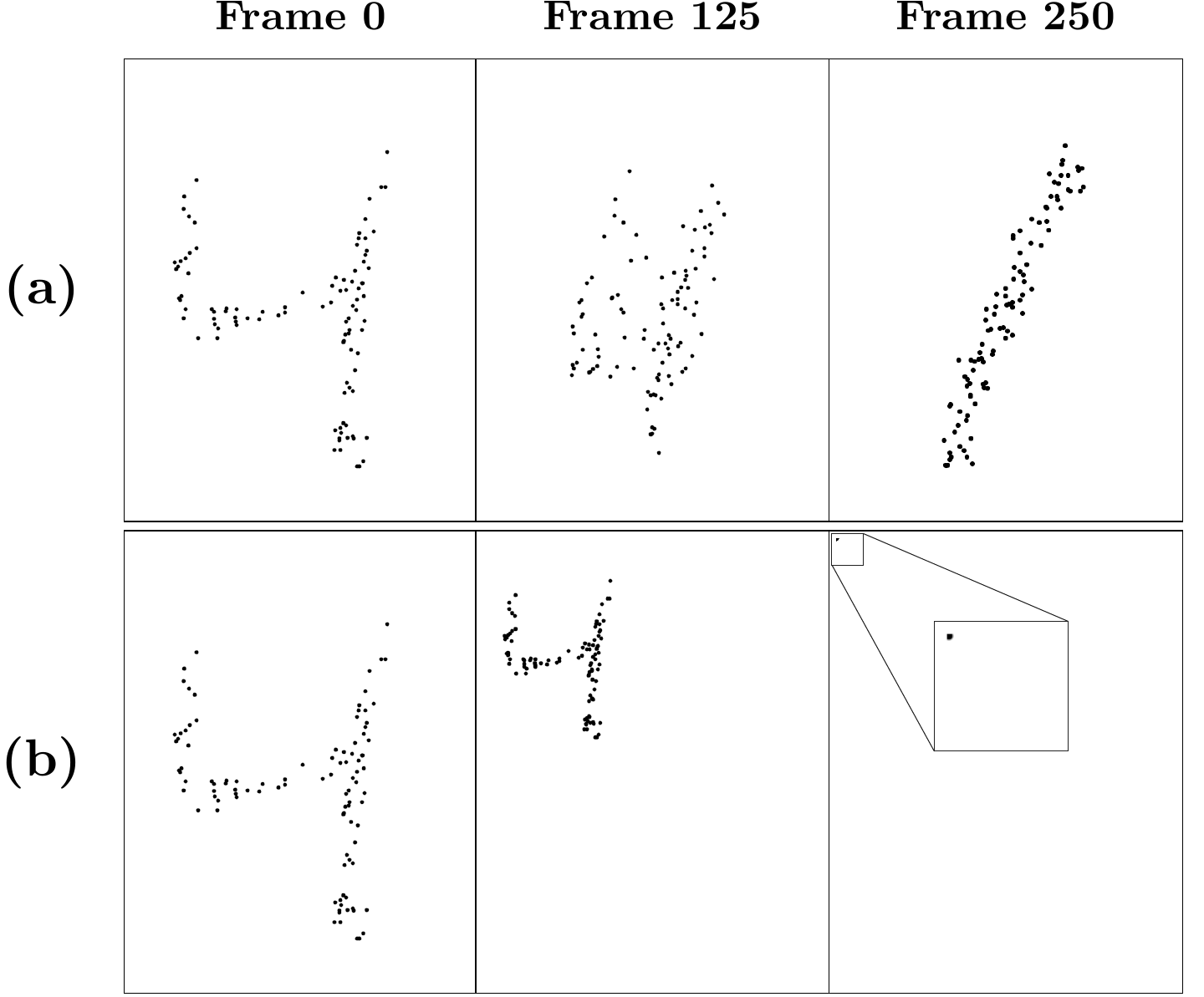}
    \caption{Image morphing via PI-PGD (a) and Sinkhorn (b). Columns show initial, middle and final frames from the animation at $t=\SI{0}{\s}$ (left), $t=\SI{4.17}{\s}$ (center) and $t=\SI{8.33}{\s}$ (right), respectively.}
\label{fig:OT_comparison_numbers}
\end{figure}

\bibliographystyle{plainurl+isbn}
\bibliography{main, new, FB, alias, IEEEabrv}
\end{document}